\documentclass[11pt]{amsart}
\usepackage{CJK}
\usepackage{indentfirst}
\usepackage{amscd}
\usepackage{amsmath, amssymb}
\usepackage{amsfonts}
\usepackage{enumerate}
\usepackage{cite}

\newcommand{\ve}{\varepsilon}
\newcommand{\ddbar}{\sqrt{-1} \partial \overline{\partial}}
\newcommand{\MA}{Monge-Amp\`{e}re}
\newcommand{\DP}{Dirichlet problem}
\begin{document}
\begin{CJK}{GBK}{song}
\newcounter{theor}
\setcounter{theor}{1}
\newtheorem{claim}{Claim}
\newtheorem{theorem}{Theorem}[section]
\newtheorem{lemma}[theorem]{Lemma}
\newtheorem{corollary}[theorem]{Corollary}
\newtheorem{proposition}[theorem]{Proposition}
\newtheorem{prop}[theorem]{Proposition}
\newtheorem{question}{question}[section]
\newtheorem{definition}[theorem]{Definition}
\newtheorem{remark}[theorem]{Remark}

\numberwithin{equation}{section}

\title[Subsolution theorem for the Monge-Amp\`{e}re equation]{Subsolution theorem for the Monge-Amp\`{e}re equation over almost Hermitian manifold}
\author[Jiaogen Zhang]{Jiaogen Zhang}
\address{Jiaogen Zhang, School of Mathematical Sciences, University of Science and Technology of China, Hefei 230026, People's Republic of China }
\email{zjgmath@mail.ustc.edu.cn}
\subjclass[2010]{32W20, 32Q60, 35B50, 31C10.}
\keywords{Complex Monge-Amp\`{e}re equation, Almost Hermitian manifold, A priori estimates, Subsolution,  $J$-plurisubharmonic}

\begin{abstract}
Let $\Omega\subseteq M$ be a bounded domain with a smooth boundary $\partial\Omega$, where $(M,J,g)$ is a compact, almost Hermitian manifold. The main result of this paper is to consider the Dirichlet problem  for a complex Monge-Amp\`{e}re equation on $\Omega$. Under the existence of a $C^{2}$-smooth strictly $J$-plurisubharmonic ($J$-psh for short) subsolution, we can solve this Dirichlet problem. Our method is based on the properties of subsolutions which have been widely used for fully nonlinear elliptic equations over Hermitian manifolds.
\end{abstract}
\maketitle

\linespread{1.2}

\section{Introduction}
Let $(M,J,g)$ be a compact almost Hermitian manifold of real dimension $2n$, and let $\Omega\subseteq M$ be a smooth domain with a smooth boundary $\partial \Omega$. In what follows, we denote by $\omega$  the K\"ahler form of $g$, i.e., 
$$\omega(X,Y)=g(JX,Y),$$ 
for all smooth vector fields $X,Y$ on $M$. We  shall consider the  subsolution theorem for the Monge-Amp\`{e}re equation
\begin{equation}\label{dp}
	\left\{\begin{array}{ll}
		(\ddbar u)^{n}=e^{h}  \omega^{n}\quad&\textrm{in}~{\Omega}; \\[1mm]
		u=\varphi &\textrm{on}~ \partial{\Omega}.\\[1mm]
	\end{array}\right.
\end{equation}

Our main result is 
\begin{theorem}
	Let  $\varphi, h\in C^{\infty}(\bar{\Omega})$ with $\inf_{\bar{\Omega}}h>-\infty$. Suppose that there exists a strictly $J$-psh subsolution $\underline{u}\in C^{2}(\bar{\Omega})$ for Eq. $($\ref{dp}$)$, that is,
	\begin{equation}\label{sub}
		\left\{\begin{array}{ll}
			(\ddbar \underline{u})^{n}\geq e^{h}  \omega^{n}\quad &\textrm{in}~{\Omega}; \\[1mm]
			\underline{u}=\varphi &\textrm{on}~ \partial{\Omega}.\\[1mm]
		\end{array}\right.
	\end{equation}
	Then there exists a unique smooth strictly $J$-psh solution $u$ for Eq. $($\ref{dp}$)$. 
\end{theorem}

The study of  the complex Monge-Amp\`{e}re equation (\ref{dp}) (on  $\mathbb{C}^{n}$) is closely related to certain problems in geometry and complex analysis; see, for instance, \cite{Chen00,Gu98,Guan02} and reference therein. The equation has been  studied extensively over the past several decades; see \cite{BT76,BT78,Blo09b,CKNS85,FSX,Gu98,Gu10,Guan02,He12,JY13,LLZ20,LLZ21,Li04,Sch84,TWWY15,Wan12} etc. Inspired by Guan's work \cite{Gu98}, it is natural to assume the existence of subsolutions in order to solve Eq. (\ref{dp}).

The purpose of this paper is to study the Dirichlet problem for the  complex Monge-Amp\`{e}re equation on a general manifold, where the almost complex structure  might not be integrable; that is, a manifold, locally, does not look like $\mathbb{C}^{n}$.  Let us remind ourselves that when the domain $\Omega\subseteq M$ admits a strictly $J$-psh defining function,  the Eq. (\ref{dp}) was already solved by Pli\'s \cite{P14}.
His resolution could be understood as a generalized version of  \cite{CKNS85}, but the underlying structure is only almost complex. Many interesting results were also obtained by Harvey-Lawson \cite{HL}. 

The  Dirichlet problems regarding other related geometric PDEs  also attracts the attension of  many mathematicians. For instance, Wang-Zhang \cite{WZ} studied the Dirichlet problem for the Hermitian-Einstein equation over an almost Hermitian manifold. In addition, the twisted quiver bundle on an almost complex manifold was researched by Zhang \cite{ZX}. Very recently, Li-Zheng \cite{LZ20} investigated the Dirichlet problem for a class of fully nonlinear elliptic equations, and obtained the boundary second order estimates. 

The structure of this paper is as follows: in Sect.~2 we  collect some basic concepts regarding almost Hermitian manifolds. In Sects.~3-5 we  give the global estimates up to the second order. Once we have these estimates in hand, higher order estimates can be also obtained by the classical Evans-Krylov theory (see, for instance, \cite{TWWY15} ) and the Schauder theory. Then we can use the standard continuity method to obtain the existence; the proof of this can be found in \cite{Gu98}, so we shall omit the standard step here. In  Sect.~6, we obtain a  strictly $J$-psh subsolution for (\ref{dp}) under the existence of a  strictly $J$-psh defining function.
\section{Preliminaries}
Let $(M,J, g)$ be a compact manifold of real dimension $2n$ with the Riemannian metric $g$ satisfying that
\[g(Ju,Jv)=g(u,v), \ \forall u,v\in TM,\]
where $J$ is the almost complex structure. Then the complexified tangent bundle can be divided as \[TM\otimes_{\mathbb{R}}\mathbb{C}=T_{0,1}M\oplus T_{1,0}M,\]
where $T_{0,1}M$ and  $T_{1,0}M $ are the  $\sqrt{-1}$ and $-\sqrt{-1}$-eigenspaces of $J$.  Similarly, the induced almost complex structure $J^{*}$ on the cotangent bundle $T^{*}M$ is defined by $J^{*}\alpha:=-\alpha\circ J$. Then we have a natural decomposition  
\[T^{*}M\otimes_{\mathbb{R}}\mathbb{C}=T^{0,1}M\oplus T^{1,0}M.\]
For brevity,  we will also denote $J^{*}$ by $J$, if no confusion occurs. For the decomposition of the $k$-th product of a complexified contangent bundle,
\[\Lambda^{k}T^{*}M\otimes_{\mathbb{R}} \mathbb{C}=\underset{p+q=k}{\bigoplus} \Lambda^{p,q}M.\]
Let $\mathcal{A}^{p,q}$ be the set of smooth sections on $\Lambda^{p,q}M$ and denote that \[\mathcal{A}^{k}:= \underset{p+q=k}{\bigoplus} \mathcal{A}^{p,q}.\] 
We consider the exterior derivative $d: \mathcal{A}^{k}\rightarrow \mathcal{A}^{k+1}$ satisfying $d^{2}=0$.
Let $\Pi_{p+1,q},\Pi_{p,q+1},\Pi_{p+2,q-1}$ and $ \Pi_{p-1,q+2}$ be the projection of $\mathcal{A}^{k+1}$ to $\mathcal{A}^{p+1,q},$ $\mathcal{A}^{p,q+1},$ $\mathcal{A}^{p+2,q-1}$ and $\mathcal{A}^{p-1,q+2}$ respectively. Thus,
\[d=\partial+\bar{\partial}+T+\overline{T},\]
where 
$$\partial=\Pi_{p+1,q}\circ d,\ \bar{\partial}=\Pi_{p,q+1}\circ d, \ T=\Pi_{p+2,q-1}\circ d, \ \overline{T}=\Pi_{p-1,q+2}\circ d.$$
In particular, if $v\in C^{2}(M,\mathbb{R})$, then $\bar{\partial}v\in \mathcal{A}^{0,1}$ and
\[d\bar{\partial}v=\partial\bar{\partial}v+\bar{\partial}^{2}v+T\bar{\partial}v.\]
Taking the complex conjugates and adding together, 
\[T\bar{\partial}v=-\partial^{2}v, \qquad \partial\bar{\partial}v=-\bar{\partial}\partial v,\]
which implies that $\ddbar v$ is a real (1,1) form on $M$. Based on the notation in \cite{Pa05,P14}, letting $e_{1},\cdots, e_{n}$ be a local $g$-orthonormal frame of $T_{1,0}M$,  we define
\[v_{i\bar{j}}:= e_{i}\bar{e}_{j}v-[e_{i},\bar{e}_{j}]^{(0,1)}v.\]
Then, in this local chart,
\begin{equation}
	\ddbar v=\sqrt{-1}\sum_{i,j=1}^{n} v_{i\bar{j}}\theta_{i}\wedge\bar{\theta}_{j},
\end{equation}
where $\theta_{1},\cdots,\theta_{n}$ is a local $g$-orthonormal frame of $T^{1,0}M$ dual to $e_{1},\cdots,e_{n}$. Thus we can rewrite the equation in (\ref{dp}) as 
\begin{equation}\label{2.1}
	\log\det(u_{i\bar{j}})=h.
\end{equation}
Let us define its linearized operator by
\[L:=u^{i\bar{j}}(e_{i}\bar{e}_{j}-[e_{i},\bar{e}_{j}]^{(0,1)}),\]
where $(u^{i\bar{j}})=(u_{i\bar{j}})^{-1}$ is the inverse matrix. Notice that $L$ is uniformly elliptic if $u\in C^{2}$ is strictly $J$-psh.
\begin{definition}
	For any $v\in C^{2}(M,\mathbb{R})$ with $\Omega\subseteq M$ being an open set,
	\begin{enumerate}
		\item[(1)] we say that $v$ is $J$-psh on $\Omega$ if the matrix $(u_{i\bar{j}})$ is nonnegative at each point of $\Omega$;
		\item[(2)] we say that $v$ is strictly $J$-psh on $\Omega$ if, for each $\varphi\in C^{2}(\Omega)$, there exists $\varepsilon_{0}>0$ such that $u+\varepsilon \varphi$ is $J$-psh on $\Omega$ for all $0<\varepsilon<\varepsilon_{0}$.
	\end{enumerate}
	We denote the set of $J$-psh functions on $\Omega$ by PSH$(\Omega)$. 
\end{definition}

Let us recall the notion of canonical connections on almost Hermitian manifolds. 

Supposing $(M,J,g)$ is an almost Hermitian manifold, there exists a canonical connection $\nabla $ on $M$ which plays a very similar role to that of the Chern connection on the Hermitian manifold. Usually, we say that a connection on $(M,J,g)$ is an almost-Hermitian connection if $\nabla g=\nabla J=0$.   Noticing that such connection always exists \cite{KN63}, we have  the following theorem (see \cite{Gau97,TWY08}):
\begin{theorem}
	There exists a unique almost-Hermitian connection $\nabla$ on an almost Hermitian manifold $(M,J,g)$ whose $(1,1)$ part of the torsion vanishes. 
\end{theorem}
This connection was found by Ehresmann-Libermann \cite{EL51}. Sometimes it is also referred to the Chern connection, because no confusion occurs when $J$ is integrable.
Under a local frame like the previous one, we have that
\begin{equation}
	\ddbar v=\sqrt{-1}\sum_{i,j=1}^{n} (\nabla_{\bar{j}}\nabla_{i}v)\theta_{i}\wedge\bar{\theta}_{j}.
\end{equation}
\subsection{Properties of subsolution}
The following lemma is due to  Guan \cite{Gu14}, who proved it for more general fully nonlinear PDEs: 
\begin{lemma}\label{Lma2.1}
	Let $\underline{u}\in C^{2}(\bar{\Omega})$ be a strictly $J$-psh subsolution to the Eq. $($\ref{dp}$)$. There exist constants $N, \theta>0$ such that if $\sum_{i=1}^{n}u_{i\bar{i}}\geq N$ at a point $p\in\Omega$ where $g_{i\bar{j}}=\delta_{ij}$ and the matrix $\{u_{i\bar{j}}\}$ is diagonal, then 
	\begin{equation}\label{2.2}
		L(\underline{u}-u)\geq \theta(\sum_{i=1}^{n} u^{i\bar{i}}+1)\qquad \textrm{in}~ \Omega.
	\end{equation}
\end{lemma}
Let us remark that since $\underline{u}$ is  strictly $J$-psh, there exists a uniform constant $\tau\in(0,1)$ such that 
\begin{equation}\label{tau}
	\ddbar \underline{u}\geq \tau \omega.
\end{equation}
\subsection{Maximum principle}
We have the following useful lemma.
\begin{lemma}\cite[p. 215]{CKNS85}\label{Lma2.2} Let $\Omega\subseteq M$ be a smooth bounded domain. If $u,v\in C^{2}(\bar{\Omega})\cap PSH(\Omega)$ with $u$ strictly J-psh and
	$\det(u_{i\bar{j}})\geq \det(v_{i\bar{j}})$, then $u-v$ attains its maximum on $\partial\Omega$.
\end{lemma}
\section{$C^{0}$ and $C^{1}$ estimates}
\subsection{Uniform estimate}
Let $\bar{u}\in C^{2}(\bar{\Omega})$ be a solution of the Dirichlet problem
\begin{equation}\label{a}
	\left\{\begin{array}{ll}
		L(u)=0&\quad \textrm{in}~{\Omega}; \\[1mm]
		u=\varphi& \quad\textrm{on}~ \partial{\Omega},\\[1mm]
	\end{array}\right.
\end{equation}
where $\bar{u}$ could be understood as the $L$-harmonic extension of $\varphi_{|\partial\Omega}$. 
\begin{lemma}
	Let $u$ $($resp. $\underline{u}$ $)$ be the solution $($resp. subsolution$)$ of Eq. $($\ref{dp}$)$. We have that
	\begin{equation}\label{}
		\underline{u}\leq u\leq \bar{u}.
	\end{equation}
\end{lemma}
\begin{proof}
	On the one hand, as $\underline{u}$ is a subsolution of (\ref{dp}), the first inequality follows from Lemma \ref{Lma2.2}. On the other hand, since $L(u)=n$, we know that $u$ is a subsolution of (\ref{a}).  By the maximum principle (for operator $L$), we also get the second inequality. 
\end{proof}
\subsection{Boundary gradient estimate}
\begin{lemma}
	Let $u$ $($resp. $\underline{u}$$)$ be a solution $($resp. subsolution$)$ of Eq. $($\ref{dp}$)$. Then there exists a constant $C=C(\|\underline{u}\|_{C^{1}(\bar{\Omega})},h,\varphi)$ such that
	\begin{equation}\label{}
		\max_{\partial\Omega} |\partial u|\leq C.
	\end{equation}
\end{lemma}

\begin{proof}
	By the previous lemma, together with the fact that $u,\underline{u}$ and $h$ have the same boundary value $\varphi_{|\partial\Omega}$, we have $|\partial u|\leq \sup\{|\partial \underline{u}|, |\partial \bar{u}|\}$ on $\partial\Omega$, and the lemma follows.
\end{proof}
\subsection{Global gradient estimate}
\begin{proposition}
	Let $u$ $($resp. $\underline{u}$$)$ be a solution $($resp. subsolution$)$ of Eq. $($\ref{dp}$)$. Then
	\begin{equation}\label{c1}
		\max_{\overline\Omega} |\partial u|\leq C
	\end{equation}
	for some positive constant $C=C(\|\underline{u}\|_{C^{1}(\bar{\Omega})}, \|u\|_{C^{0}(\bar{\Omega})}, \|u\|_{C^{0,1}(\partial{\Omega})},h)$.
\end{proposition}

\begin{proof}
	Let $\vartheta= \frac{1}{3}e^{B\eta}$ for $\eta=\underline{u}-u+\sup_{\bar{\Omega}}(u-\underline{u}),$ where $ B>0$ is a constant to be picked up later. We will prove (\ref{c1}) by applying the maximum principle to 
	\[V:= e^{\vartheta}|\partial u|^{2}.\]
	Suppose that $V$ achieves its maximum at $x_{0}\in \textrm{Int}(\Omega)$. Near $x_{0}$, we choose a local $g$-unitary frame  $(e_{1},\cdots,e_{n})$ such that $g_{i\bar{j}}=\delta_{ij}$. Moreover,  the matrix $(u_{i\bar{j}})$ is diagonal at $x_{0}$.
	
	At  $x_{0}$, it follows from the maximum principle that
	\begin{equation}\label{3.9}
		\begin{split}
			0\geq &\frac{L(V)}{B\vartheta e^{\vartheta}|\partial u|^{2}}
			=\frac{L(e^{\vartheta})}{B\vartheta e^{\vartheta}}+\frac{L(|\partial u|^{2})}{B\vartheta|\partial u|^{2}}+2u^{i\bar{i}}\textrm{Re}\big(e_{i}({\vartheta})
			\frac{\bar{e}_{i}(|\partial u|^{2})}{B\vartheta|\partial u|^{2}}\big)\\
			=&L(\eta)+B(1+\vartheta)u^{i\bar{i}}|\eta_{i}|^{2}+\frac{L(|\partial u|^{2})}{B\vartheta|\partial u|^{2}}+\frac{1}{|\partial u|^{2}}\cdot \big((*)+(**)\big),\\
		\end{split}
	\end{equation}
	where 
	\begin{alignat}{2} 
		& (*)
		& & :=2\sum_{j=1}^{n}u^{i\bar{i}}\textrm{Re}\big(e_{i}(\eta)\bar{e}_{i}e_{j}(u)\bar{e}_{j}(u)\big); \\ 
		&(**)
		& & :=2\sum_{j=1}^{n}u^{i\bar{i}}\textrm{Re}\big(e_{i}(\eta)\bar{e}_{i}\bar{e}_{j}(u)e_{j}(u)\big).
	\end{alignat}
	\footnote{The constants $C, C'$ in the rest of the section are distinct, where $C$ is a constant depending on all the allowed data, but  $C'$ further depends on a constant $B$ that we are yet to choose.}By a straightforward calculation,
	\[
	L(|\partial u|^{2})=u^{i\bar{i}}\big({e_{i}e_{\bar{i}}}(|\partial u|^{2})-[e_{i},\bar{e}_{i}]^{0,1}(|\partial u|^{2})\big):= I+II+III,
	\]
	where
	\begin{alignat}{2} 
		& I
		& & :=u^{i\bar{i}}(e_{i}\bar{e}_{i}e_{j}u-[e_{i},\bar{e}_{i}]^{0,1}e_{j}u)\bar{e}_{j}u; \\ 
		&II
		& & :=u^{i\bar{i}}(e_{i}\bar{e}_{i}\bar{e}_{j}u-[e_{i},\bar{e}_{i}]^{0,1}\bar{e}_{j}u)e_{j}u;\\
		&III
		&& :=u^{i\bar{i}}(|e_{i}e_{j}u|^{2}+|e_{i}\bar{e}_{j}u|^{2}).
	\end{alignat}
	Differentiating (\ref{2.1}) along $e_{j}$, 
	\[u^{i\bar{i}}(e_{j}e_{i}\bar{e}_{i}u-e_{j}[e_{i},\bar{e}_{i}]^{0,1}u)=h_{j}.
	\]
	Notice that
	\[
	\begin{split}
		& u^{i\bar{i}}(e_{i}\bar{e}_{i}e_{j}u-[e_{i},\bar{e}_{i}]^{0,1}e_{j}u) \\
		= &  u^{i\bar{i}}(e_{j}e_{i}\bar{e}_{i}u+e_{i}[\bar{e}_{i},e_{j}]u
		+[e_{i},e_{j}]\bar{e}_{i}u-[e_{i},\bar{e}_{i}]^{0,1}e_{j}u)\\
		= &h_{j}+u^{i\bar{i}}e_{j}[e_{i},\bar{e}_{i}]^{0,1}u
		+u^{i\bar{i}}(e_{i}[\bar{e}_{i},e_{j}]u
		+[e_{i},e_{j}]\bar{e}_{i}u-[e_{i},\bar{e}_{i}]^{0,1}e_{j}u)\\
		=&h_{j}+ u^{i\bar{i}}\big\{e_{i}[\bar{e}_{i},e_{j}]u+\bar{e}_{i}[e_{i},e_{j}]u+[[e_{i},e_{j}],\bar{e}_{i}]u
		-[[e_{i},\bar{e}_{i}]^{0,1},e_{j}]u\big\}
		.\end{split}
	\]
	We may assume that $|\partial u|\gg 1$ (otherwise we are done), and set $$\mathcal{U}:=\sum_{i=1}^{n}u^{i\bar{i}}.$$ 
	By the Cauchy-Schwarz inequality, for each $0<\ve\leq \frac{1}{2}$,
	\begin{equation}\label{4.9}
		\begin{split}
			I+II
			\geq& 2\textrm{Re}\big(\sum_{j=1}^{n}h_{j}u_{\bar{j}}\big)-C|\partial u|\sum_{j=1}^{n}u^{i\bar{i}}(|e_{i}e_{j}u|+|e_{i}\bar{e}_{j}u|)-C|\partial u|^{2} \mathcal{U}\\
			\geq & 2\textrm{Re}\big(\sum_{j=1}^{n}h_{j}u_{\bar{j}}\big)-\frac{C}{\varepsilon}|\partial u|^{2}\mathcal{U}-\ve\sum_{j=1}^{n} u^{i\bar{i}}(|e_{i}e_{j}u|^{2}+|e_{i}\bar{e}_{j}u|^{2}).
		\end{split}
	\end{equation}
	It then follows from (\ref{3.9}) that
	\begin{equation}\label{3.7}
		\begin{split}
			\frac{L(|\partial u|^{2})}{B\vartheta|\partial u|^{2}}\geq& \frac{-C}{B\vartheta|\partial u|}+ (1-\varepsilon) \sum_{j=1}^{n}u^{i\bar{i}}\frac{|e_{i}e_{j}u|^{2}+|e_{i}\bar{e}_{j}u|^{2}}{B\vartheta|\partial u|^{2}}- \frac{C\mathcal{U}}{B\vartheta\varepsilon} .
		\end{split}
	\end{equation}
	As $0<\varepsilon\leq \frac{1}{2}$, $1\leq (1-\varepsilon)(1+2\varepsilon)$. Thus, 
	\begin{equation}\label{3..9}
		\begin{split}
			(*)=& 2\sum_{j=1}^{n}u^{i\bar{i}}\textrm{Re}\Big(e_{i}(\eta)
			\bar{e}_{j}(u)\big\{e_{j}\bar{e}_{i}(u)-[e_{j},\bar{e}_{i}]^{0,1}(u)
			-[e_{j},\bar{e}_{i}]^{1,0}(u)\big\}\Big)\\
			=&   2\textrm{Re}\big(\sum_{j=1}^{n}\eta_{j}u_{\bar{j}}\big)-2\sum_{j=1}^{n}u^{i\bar{i}}\textrm{Re}\big(e_{i}(\eta)
			\bar{e}_{j}(u)[e_{j},\bar{e}_{i}]^{1,0}(u)\big)\\
			\geq&  2\textrm{Re}\big(\sum_{j=1}^{n}\eta_{j}u_{\bar{j}}\big)-\varepsilon B\vartheta|\partial u|^{2}u^{i\bar{i}}|\eta_{i}|^{2}
			-\frac{C}{B\vartheta\varepsilon}|\partial u|^{2}\mathcal{U};
		\end{split}
	\end{equation}
	\begin{equation}\label{3..10}
		\begin{split}
			(**)\geq & -\frac{(1-\varepsilon)}{B\vartheta}\sum_{j=1}^{n}u^{i\bar{i}}|\bar{e}_{i}\bar{e}_{j}(u)|^{2}
			-(1+2\varepsilon)B\vartheta|\partial u|^{2}u^{i\bar{i}}|\eta_{i}|^{2}.
		\end{split}
	\end{equation}
	It follows from (\ref{3..9}) and (\ref{3..10}) that
	\begin{equation}\label{3.12}
		\begin{split}
			\frac{1}{|\partial u|^{2}}\cdot \big((*)+(**)\big)
			\geq &\frac{2\textrm{Re}\big(\sum_{j=1}^{n}\eta_{j}u_{\bar{j}}\big)}{|\partial u|^{2}}-(1+3\ve)B\vartheta u^{i\bar{i}}|\eta_{i}|^{2}
			\\&-\frac{C}{B\vartheta\varepsilon}\mathcal{U}
			-(1-\varepsilon)\sum_{j=1}^{n}u^{i\bar{i}}\frac{|\bar{e}_{i}\bar{e}_{j}(u)|^{2}}{B\vartheta|\partial u|^{2}}.
		\end{split}
	\end{equation}
	Combining (\ref{3.9}), (\ref{3.7}) and (\ref{3.12}) gives us
	\[
	0\geq L(\eta)+B(1-3\ve w) u^{i\bar{i}}|\eta_{i}|^{2}-\frac{C\mathcal{U}}{B\vartheta\varepsilon}
	-\frac{C}{B\vartheta|\partial u|}
	+\frac{2\textrm{Re}\big(\sum_{j=1}^{n}\eta_{j}u_{\bar{j}}\big)}{|\partial u|^{2}}.
	\]
	Hence, if we choose $\ve=\frac{1}{6\vartheta(x_0)}\leq \frac{1}{2}$, 
	\begin{equation}\label{3.11}
		\begin{split}
			&L(\eta)+ \frac{2\textrm{Re}\big(\sum_{j=1}^{n}\eta_{j}u_{\bar{j}}\big)}{|\partial u|^{2}}
			+\frac{B}{2}
			u^{i\bar{i}}|\eta_{i}|^{2}
			\leq \frac{C}{B\vartheta|\partial u|}+ \frac{C}{B}\mathcal{U}.
		\end{split}
	\end{equation}
	~\\
	\textbf{Case 1.} $\sum_{i=1}^{n}u_{i\bar{i}}\geq N$ for some  \textit{N}  as in Lemma \ref{Lma2.1}. We divide the proof into two parts.
	
	\textbf{Subcase 1(i)} If $u^{j\bar{j}}\geq D$ for some $j$, where $D>0$ is a large constant to be determined shortly. Thus
	\[
	L(\eta)\geq\theta+\theta\mathcal{U}\geq \theta+\frac{D\theta}{2}+\frac{\theta}{2}\mathcal{U}.
	\]
	We may assume that $|\partial u|\geq |\partial\underline{u}|$, whence
	$|\partial \eta|\leq 2|\partial u|$, then
	\begin{equation}\label{case 1a}
		\frac{2\textrm{Re}\big(\sum_{j=1}^{n}\eta_{j}u_{\bar{j}}\big)}{|\partial u|^{2}}\geq -4.	
	\end{equation}
	Substituting this into (\ref{3.11}),
	\begin{equation*}\label{}
		\theta+\frac{D\theta}{2}-4+\big(\frac{\theta}{2}-\frac{C}{B}\big)\mathcal{U}\leq  \frac{C}{B\vartheta|\partial u|}.
	\end{equation*}
	~\\
	We may choose $B, D$ sufficiently large such that $\theta\geq \frac{C}{B}$ and $D\theta\geq 8$, whence (\ref{c1}) follows.
	
	\textbf{Subcase 1(ii)} If $u^{j\bar{j}}\leq D$ for each $j=1,2,\cdots,n$, since $|\partial u|\geq \max\{1,|\partial \underline{u}|\}$, 
	\[
	\begin{split}
		\frac{2\textrm{Re}\big(\sum_{j=1}^{n}\eta_{j}u_{\bar{j}}\big)}{|\partial u|^{2}} \geq    -\frac{B}{4}u^{i\bar{i}}|\eta_{i}|^{2}
		-\frac{4}{B|\partial u|^{2}}\sum_{i=1}^{n}u_{i\bar{i}},
	\end{split}
	\]
	and it follows from (\ref{3.11}) that
	\[
	\theta
	+\theta\mathcal{U}\leq  \frac{C}{B\vartheta|\partial u|}+ \frac{C}{B}\mathcal{U}+\frac{4}{B|\partial u|^{2}}\sum_{i=1}^{n}u_{i\bar{i}}.
	\]
	Notice that $\theta\geq \frac{C}{B}$. Thus,
	\begin{equation}\label{3.14}
		\theta\leq \frac{C}{B\vartheta|\partial u|}+\frac{4}{B|\partial u|^{2}}\sum_{i=1}^{n}u_{i\bar{i}}.
	\end{equation}
	It is useful to order $\{u_{i\bar{i}}\}_{i=1}^{n}$ such that $u_{1\bar{1}}\geq \cdots\geq u_{n\bar{n}}$ at $x_{0}$.
	Thus,
	$u_{1\bar{1}}D^{-(n-1)}\leq \prod_{i=1}^{n}u_{i\bar{i}}=e^{h}.$
	Then we have
	\[
	\sum_{i=1}^{n}u_{i\bar{i}}\leq n u_{1\bar{1}}\leq ne^{\sup_{\bar{\Omega}}h}D^{n-1}.
	\]
	Substituting this into  (\ref{3.14}), we get that $|\partial u|\leq C'$.
	
	~\\
	\textbf{Case 2.}  $\sum_{i=1}^{n}u_{i\bar{i}}\leq N$, so $u^{k\bar{k}}\geq N^{-1}$ for each $k$.
We have that
	\begin{equation}\label{case 2-1}
		u^{i\bar{i}}|\eta_{i}|^{2}\geq N^{-1}|\partial \eta|^{2}.	
	\end{equation}
The fact that $\underline{u}$ is strictly $J$-psh implies that
	\begin{equation}\label{case 2-2}
		L(\eta)\geq \tau \mathcal{U}-n.
	\end{equation} 
	It follows from (\ref{3.11}), (\ref{case 1a}), (\ref{case 2-1}) and (\ref{case 2-2}) that
	\[
	(\tau-\frac{C}{B}) \mathcal{U}+
	B N^{-1}|\partial \eta|^{2}\leq \frac{C}{B\vartheta|\partial \eta|}+5n,
	\]
	since $|\partial u|\geq \max\{1,|\partial \underline{u}|\}$. We further assume that $\tau\geq \frac{C}{B}$, so
	\[
	B N^{-1}|\partial \eta|^{2}\leq \frac{C}{B\vartheta|\partial \eta|}+5n,
	\]
	which implies $|\partial\eta|\leq C'$,  whence
	(\ref{c1}) follows.
\end{proof}
\section{Interior $C^{2}$ estimate}
In this section we follow the arguments of \cite{CTW} to estimate the largest eigenvalue $\lambda_{1}(\hat{\nabla}^{2}u)$ of the real Hessian $\hat{\nabla}^{2}u$, where $\hat{\nabla}$ is the Levi-Civita connection on $M$.
\begin{theorem}\label{Thm4.1}
	Let $u$ $($resp. $\underline{u}$$)$ be a solution $($resp. subsolution$)$ of Eq. $($\ref{dp}$)$. We have
	\begin{equation}\label{4.1}
		\max_{\bar{\Omega}}\lambda_{1}(\hat{\nabla}^{2}u)\leq C(1+\underset{\partial{\Omega}}{\max}|\ddbar u|),
	\end{equation}
	where $C$ is a constant depending  on $\|h\|_{C^{2}(\Omega)}, \|u\|_{C^{1}(\bar{\Omega})}$ and $\|\underline{u}\|_{C^{2}(\bar{\Omega})}$.
\end{theorem}
\begin{proof}
For brevity, we denote  $\varpi:= \underline{u}-u+\sup_{\bar{\Omega}}(u-\underline{u})+1$. Define
\[\mathcal{Q}:=\log \lambda_{1}(\hat{\nabla}^{2}u)+\phi(|\partial u|^{2})+e^{B\varpi}\]
in $\Omega':=\{\lambda_{1}(\hat{\nabla}^{2}u)>0\}\subseteq \Omega$, where $B$ is a large constant to be determined later, and $\phi$ is defined by
\[\phi(s):=-\frac{1}{2}\log(1+\sup_{\bar{\Omega}}|\partial u|^{2}-s).\]
Setting $K:=1+\sup_{\bar{\Omega}}|\partial u|^{2}$, we have that
\[
\frac{1}{2K}\leq \phi'(|\partial u|^{2})\leq \frac{1}{2}, \qquad \phi''=2(\phi')^{2}.
\]
We may assume that $\Omega'$ is a nonempty (relative) open set, otherwise we are done. As $z$ approaches $\partial{\Omega'}\setminus\partial{\Omega}$, $\mathcal{Q}\rightarrow -\infty$, if $\mathcal{Q}$ achieves its maximum on $\partial{\Omega}$, then we are done, by (\ref{4.1}). Thus, we may assume that $\mathcal{Q}$ achieves its maximum  in $\textrm{Int}(\Omega')$.  Near $x_{0}$, we choose a local $g$-unitary frame $(e_{1},\cdots, e_{n})$ such that, at $x_{0}$,
\begin{equation}\label{5..6}
	\text{$g_{i\overline{j}}=\delta_{ij}$, $u_{i\overline{j}}=\delta_{ij}u_{i\overline{i}}$ and $u_{1\overline{1}}\geq u_{2\overline{2}}\geq\cdots\geq u_{n\overline{n}}$.}
\end{equation}
In addition, there exists a normal coordinate system $(U,\{x^{\alpha}\}_{i=1}^{2n})$ in a neighbourhood of $x_{0}$ such that
\begin{equation}\label{real frame and complex frame}
	e_{i}=\frac{1}{\sqrt{2}}(\partial_{2i-1}-\sqrt{-1}\partial_{2i}) \qquad \text{for} \ i=1,\cdots,n;
\end{equation}
\begin{equation}\label{first derivative of background metric}
	\frac{\partial g_{\alpha\beta}}{\partial x^{\gamma}}=0 \qquad \text{for} \ \alpha,\beta,\gamma=1,\cdots,2n,
\end{equation}
where $g_{\alpha\beta}:=g(\partial_{\alpha}, \partial_{\beta}).$

We  define an endomorphism $\Phi=(\Phi_{\beta}^{\alpha})$ of $TM$ by
\[
\Phi_{\beta}^{\alpha}:=g^{\alpha\gamma}(\hat{\nabla}_{\gamma\beta}^{2}u-S_{\gamma\beta})
\]
for some smooth section $S$ on $T^{*}M\otimes T^{*}M$ such that
\[
\lambda_{1}(\Phi)\leq \lambda_{1}(\hat{\nabla}^{2}u)\qquad \textrm{in} ~\Omega',
\]
with the equality only at $x_{0}$, but also $\lambda_{1}(\Phi)\in C^{2}(\Omega)$ (cf. \cite{CTW,Gab}).  For any  $\beta$, let $V_{\beta}$ be eigenvector of  $\Phi$ with an eigenvalue $\lambda_{\beta}$.
The proof needs the following derivatives of $\lambda_{1}$, which can be found in \cite{CTW,Gab,Spr}:
\begin{lemma}\label{the derivative of eigenvalues}
	At $x_{0}$, we have that
	\begin{equation}
		\begin{split}
			\frac{\partial \lambda_{1}}{\partial \Phi^{\alpha}_{\beta} }&=V_{1}^{\alpha}V_{1}^{\beta};\\
			\frac{\partial^{2} \lambda_{1}}{\partial \Phi^{\alpha}_{\beta}\partial \Phi^{\gamma}_{\delta}}&
			=\sum_{\kappa>1}\frac{1}{\lambda_{1}-\lambda_{\kappa}}\big(V_{1}^{\alpha}V_{\kappa}^{\beta}V_{\kappa}^{\gamma}V_{1}^{\delta}
			+V_{\kappa}^{\alpha}V_{1}^{\beta}V_{1}^{\gamma}V_{\kappa}^{\delta}\big).
		\end{split}
	\end{equation}
\end{lemma}

We will prove (\ref{4.1}) by applying the maximum principle to the quantity
\[
Q:=\log\lambda_{1}(\Phi)+\phi(|\partial u|^{2})+\phi(\varpi).
\]
Clearly, $Q$ attains its maximum at $x_{0}$. Thus, at $x_{0}$, 
\begin{equation}\label{5..10}
	\frac{1}{\lambda_{1}}e_{i}(\lambda_{1})=-\phi'e_{i}(|\partial u|^{2})-Be^{B\varpi}\varpi_{i}, \qquad \text{for all} \ 1\leq i\leq n;
\end{equation}
\begin{equation}\label{}
	\begin{split}
		0\geq L(Q)=& \frac{L(\lambda_{1})}{\lambda_{1}}- u^{i\bar{i}}\frac{|e_{i}(\lambda_{1})|^{2}}{\lambda_{1}^{2}}  +\phi'' u^{i\bar{i}}|e_{i}(|\partial u|^{2})|^{2} \\
		&+\phi'L(|\partial u|^{2})+Be^{B\varpi}L(\varpi)+B^{2}e^{B\varpi} u^{i\bar{i}}|\varpi_{i}|^{2}.
	\end{split}
\end{equation}
For the rest of this section we may assume that $\sum_{i=1}^{n}u_{i\bar{i}}\geq N$ for the constant $N$  in Lemma \ref{Lma2.1} (otherwise we are done).
\subsection{Lower bound of $L(Q)$}
\begin{proposition}\label{Lma4.2}
	For each $\ve\in (0,\frac{1}{2}]$, at $x_{0}$, we have that
	\begin{equation}\label{4.7"}
		\begin{split}
			0\geq &L(Q)\\ \geq & (2-\ve)\sum_{\alpha>1}u^{i\bar{i}}\frac{|e_{i}(u_{V_{\alpha}V_{1}})|^{2}}{\lambda_{1}(\lambda_{1}-\lambda_{\alpha})}
			+\frac{1}{\lambda_{1}}u^{i\bar{i}}u^{k\bar{k}}|V_{1}(u_{i\bar{k}})|^{2}\\&
			-(1+\ve)u^{i\bar{i}}\frac{|e_{i}(\lambda_{1})|^{2}}{\lambda_{1}^{2}}
			-\frac{C}{\ve}\mathcal{U}
			+ \frac{\phi'}{2}\sum_{j=1}^{n} u^{i\bar{i}}(|e_{i}e_{j}u|^{2}+|e_{i}\bar{e}_{j}u|^{2})\\&+\phi'' u^{i\bar{i}}|e_{i}(|\partial u|^{2})|^{2}+Be^{B\varpi}L(\varpi)+B^{2}e^{B\varpi} u^{i\bar{i}}|\varpi_{i}|^{2}.\\
		\end{split}
	\end{equation}
\end{proposition}
\begin{proof}
	First, we calculate $L(\lambda_{1})$. Let $$u_{ij}=e_{i}e_{j}u-(\hat{\nabla}_{e_{i}}e_{j})u, 
	\qquad u_{V_{i}V_{j}}=u_{kl}V^{k}_{i}V^{l}_{j}.$$ 
	By Lemma \ref{the derivative of eigenvalues} and (\ref{first derivative of background metric}), we can infer that
	\begin{equation}\label{5..13}
		\begin{split}
			L(\lambda_{1})=&u^{i\bar{i}}\frac{\partial^{2} \lambda_{1}}{\partial \Phi^{\alpha}_{\beta}\partial \Phi^{\gamma}_{\delta}}
			e_{i}(\Phi^{\gamma}_{\delta})\bar{e}_{i}(\Phi^{\alpha}_{\beta})+u^{i\bar{i}}\frac{\partial \lambda_{1}}{\partial
				\Phi^{\alpha}_{\beta}}(e_{i}\bar{e}_{i}-[e_{i},\bar{e}_{i}]^{0,1})(\Phi^{\alpha}_{\beta})\\
			=&u^{i\bar{i}}\frac{\partial^{2} \lambda_{1}}{\partial \Phi^{\alpha}_{\beta}\partial \Phi^{\gamma}_{\delta}}
			e_{i}(u_{\gamma\delta})\bar{e}_{i}(u_{\alpha\beta})+u^{i\bar{i}}\frac{\partial \lambda_{1}}{\partial
				\Phi^{\alpha}_{\beta}}(e_{i}\bar{e}_{i}-[e_{i},\bar{e}_{i}]^{0,1})(u_{\alpha\beta})
			+u^{i\bar{i}}\frac{\partial \lambda_{1}}{\partial
				\Phi^{\alpha}_{\beta}}u_{\gamma\beta}e_{i}\bar{e}_{i}(g^{\alpha\gamma})\\
			\geq & 2\sum_{\alpha>1}u^{i\bar{i}}\frac{|e_{i}(u_{V_{\alpha}V_{1}})|^{2}}{\lambda_{1}-\lambda_{\alpha}}
			+ u^{i\bar{i}}(e_{i}\bar{e}_{i}-[e_{i},\bar{e}_{i}]^{0,1})(u_{V_{1}V_{1}})-C\lambda_{1}\mathcal{U}.\\
		\end{split}
	\end{equation}
	Applying $V_{1}$ to Eq. (\ref{2.1}) twice, 
	\begin{equation}\label{5..15}
		u^{i\bar{i}}V_{1}V_{1}(u_{i\bar{i}})= u^{i\bar{i}}u^{k\bar{k}}|V_{1}(u_{i\bar{k}})|^{2}+V_{1}V_{1}(h).
	\end{equation}
	\begin{lemma} If $\lambda_{1}\gg 1$, then
		\begin{equation}\label{claim2}
			\begin{split}
				&u^{i\bar{i}}(e_{i}\bar{e}_{i}-[e_{i},\bar{e}_{i}]^{0,1})(u_{V_{1}V_{1}})\\
				\geq &   u^{i\bar{i}}u^{k\bar{k}}|V_{1}(u_{i\bar{k}})|^{2}-C\lambda_{1}\mathcal{U}-2 u^{i\bar{i}}\{[V_{1},\bar{e}_{i}]V_{1}e_{i}(u)+[V_{1},e_{i}]V_{1}\bar{e}_{i}(u)\}. \\
			\end{split}
		\end{equation}
	\end{lemma}
	\begin{proof}
		By a direct calculation,
		\[
		\begin{split}
			& u^{i\bar{i}}(e_{i}\bar{e}_{i}-[e_{i},\bar{e}_{i}]^{0,1})(u_{V_{1}V_{1}})\\
			=& u^{i\bar{i}}e_{i}\bar{e}_{i}(V_{1}V_{1}(u)-(\hat{\nabla}_{V_{1}}V_{1})u)
			-u^{i\bar{i}}[e_{i},\bar{e}_{i}]^{0,1}(V_{1}V_{1}(u)-(\hat{\nabla}_{V_{1}}V_{1})u)\\
			\geq& u^{i\bar{i}}V_{1}V_{1}(e_{i}\bar{e}_{i}(u)-[e_{i},\bar{e}_{i}]^{0,1}(u))
			-2u^{i\bar{i}}\{[V_{1},\bar{e}_{i}]V_{1}e_{i}(u)+[V_{1},e_{i}]V_{1}\bar{e}_{i}(u)\}\\
			&-u^{i\bar{i}}(\hat{\nabla}_{V_{1}}V_{1})e_{i}\bar{e}_{i}(u)
			+u^{i\bar{i}}(\hat{\nabla}_{V_{1}}V_{1})[e_{i},\bar{e}_{i}]^{0,1}(u)
			-C\lambda_{1}\mathcal{U}\\
			\geq&u^{i\bar{i}}V_{1}V_{1}(u_{i\bar{i}})
			-2u^{i\bar{i}}\{[V_{1},\bar{e}_{i}]V_{1}e_{i}(u)+[V_{1},e_{i}]V_{1}\bar{e}_{i}(u)\}\\
			& +(\hat{\nabla}_{V_{1}}V_{1})(h)-C\lambda_{1}\mathcal{U}.\\
		\end{split}
		\]
		Then the lemma follows from (\ref{5..15}) if $\lambda_{1}\gg 1$.
	\end{proof}
	It follows from   (\ref{5..13}) and (\ref{claim2})  that
	\begin{equation}
		\begin{split}
			L(\lambda_{1})\geq & 2\sum_{\alpha>1}u^{i\bar{i}}\frac{|e_{i}(u_{V_{\alpha}V_{1}})|^{2}}{\lambda_{1}-\lambda_{\alpha}}
			+u^{i\bar{i}}u^{k\bar{k}}|V_{1}(u_{i\bar{k}})|^{2}
			\\&-2u^{i\bar{i}}\textrm{Re}\big([V_{1},e_{i}]V\bar{e}_{i}(u)
			+[V_{1},\bar{e}_{i}]Ve_{i}(u)\big)-C\lambda_{1}\mathcal{U}.
		\end{split}
	\end{equation}
	By (\ref{3.7}), we have that
	\begin{equation}
		L(|\partial u|^{2})\geq  \frac{1}{2}\sum_{j=1}^{n} u^{i\bar{i}}(|e_{i}e_{j}u|^{2}+|e_{i}\bar{e}_{j}u|^{2})- C\mathcal{U}.
	\end{equation}
	Thus,
	\begin{equation}\label{4.10-1}
		\begin{split}
			L(Q)
			\geq & 2\sum_{\alpha>1}u^{i\bar{i}}\frac{|e_{i}(u_{V_{\alpha}V_{1}})|^{2}}{\lambda_{1}(\lambda_{1}-\lambda_{\alpha})}
			+\frac{1}{\lambda_{1}}u^{i\bar{i}}u^{k\bar{k}}|V_{1}(u_{i\bar{k}})|^{2}+B^{2}e^{B\varpi} u^{i\bar{i}}|\varpi_{i}|^{2}
			\\&+Be^{B\varpi}L(\varpi)-2 u^{i\bar{i}}\frac{\textrm{Re}\big([V_{1},e_{i}]V_{1}\bar{e}_{i}(u)+[V_{1},\bar{e}_{i}]V_{1}e_{i}(u)\big)}{\lambda_{1}}
			-C\mathcal{U}\\
			&-u^{i\bar{i}}\frac{|e_{i}(\lambda_{1})|^{2}}{\lambda_{1}^{2}}
			+ \frac{\phi'}{2}\sum_{j=1}^{n} u^{i\bar{i}}(|e_{i}e_{j}u|^{2}+|e_{i}\bar{e}_{j}u|^{2})+\phi'' u^{i\bar{i}}|e_{i}(|\partial u|^{2})|^{2}.\\
		\end{split}
	\end{equation}
	\begin{lemma} For each $0<\ve\leq 1/2$, we have that
		\begin{equation}\label{4.10}
			\begin{split}
				& 2 u^{i\bar{i}}\frac{\textrm{Re}\big([V_{1},e_{i}]V_{1}\bar{e}_{i}(u)+[V_{1},\bar{e}_{i}]V_{1}e_{i}(u)\big)}{\lambda_{1}}\\
				\leq & \ve u^{i\bar{i}}\frac{|e_{i}(\lambda_{1})|^{2}}{\lambda_{1}^{2}}+
				\ve\sum_{\alpha>1}u^{i\bar{i}}\frac{|e_{i}(u_{V_{\alpha}V_{1}})|^{2}}{\lambda_{1}(\lambda_{1}-\lambda_{\alpha})}
				+\frac{C}{\ve}\mathcal{U}.
			\end{split}
		\end{equation}
	\end{lemma}
	\begin{proof}
		Assume that 
		\[
		[V_{1},e_{i}]=\sum_{\beta=1}^{2n} \mu_{i\beta}V_{\beta},\qquad [V_{1},\bar{e}_{i}]=\sum_{\beta=1}^{2n} \overline{\mu_{i\beta}}V_{\beta},
		\]
		where $\mu_{i\beta}\in\mathbb{C}$ are uniformly bounded constants.
		Then,
		\begin{equation}\label{three dv}
			\textrm{Re}\big([V_{1},e_{i}]V_{1}\bar{e}_{i}(u)+[V_{1},\bar{e}_{i}]V_{1}e_{i}(u)\big)\leq C\sum_{\beta=1}^{2n}|V_{\beta}V_{1}e_{i}(u)|.
		\end{equation}
		This reduces to estimate $\frac{1}{\lambda_{1}}\underset{\beta}{\sum} u^{i\bar{i}}|V_{\beta}V_{1}e_{i}(u)|$.
		Recalling the definition of Lie bracket $e_{i}e_{j}-e_{j}e_{i}=[e_{i},e_{j}]$, we have that
		\[
		\begin{split}
			\big|V_{\beta}V_{1}e_{i}(u)\big|= &\big|e_{i}V_{\beta}V_{1}(u)+V_{\beta}[V_{1},e_{i}](u)+[V_{\beta},e_{i}]V_{1}(u)\big|  \\
			=& \big|e_{i}(u_{V_{\beta}V_{1}})+e_{i}(\nabla_{V_{\beta}}V_{1})(u)+V_{\beta}[V_{1},e_{i}](u)
			+[V_{\beta},e_{i}]V_{1}(u)\big|\\
			\leq &\big|e_{i}(u_{V_{\beta}V_{1}})\big|+C\lambda_{1}.
		\end{split}
		\]
		Therefore,
		\begin{equation}\label{third order derivative 1}
			\begin{split}
				\sum_{\beta=1}^{2n} u^{i\bar{i}}\frac{|V_{\beta}V_{1}e_{i}(u)|}{\lambda_{1}}
				\leq &  \sum_{\beta=1}^{2n} u^{i\bar{i}}\frac{|e_{i}(u_{V_{\beta}V_{1}})|}{\lambda_{1}}+C\mathcal{U}\\
				= & u^{i\bar{i}}\frac{|e_{i}(\lambda_{1})|}{\lambda_{1}}+\sum_{\beta>1} u^{i\bar{i}}\frac{|e_{i}(u_{V_{\beta}V_{1}})|}{\lambda_{1}}+C\mathcal{U}.\\
			\end{split}
		\end{equation}
		For each $\ve\in (0,\frac{1}{2}]$, we deduce that
		\begin{equation}
			u^{i\bar{i}}\frac{|e_{i}(\lambda_{1})|}{\lambda_{1}}\leq \ve u^{i\bar{i}}\frac{|e_{i}(\lambda_{1})|^{2}}{\lambda_{1}^{2}}+\frac{C}{\ve}\mathcal{U};
		\end{equation}
		\begin{equation}
			\begin{split}
				\sum_{\beta>1} u^{i\bar{i}}\frac{|e_{i}(u_{V_{\beta}V_{1}})|}{\lambda_{1}}\leq & \ve\sum_{\beta>1} u^{i\bar{i}}\frac{|e_{i}(u_{V_{\beta}V_{1}})|^{2}}{\lambda_{1}(\lambda_{1}-\lambda_{\beta})}
				+\sum_{\beta>1}\frac{\lambda_{1}-\lambda_{\beta}}{\ve\lambda_{1}}\mathcal{U} \\
				\leq & \ve\sum_{\beta>1} u^{i\bar{i}}\frac{|e_{i}(u_{V_{\beta}V_{1}})|^{2}}{\lambda_{1}(\lambda_{1}-\lambda_{\beta})}
				+\frac{C}{\ve}\mathcal{U},
			\end{split}	
		\end{equation}
		where in the last inequality we have used is
		$\sum_{\beta=1}^{2n}\lambda_{\beta}=\Delta u=\Delta^{\mathbb{C}}u+T(du)\geq -C$; see \cite{CTW}.
		Here  $T$ is the torsion vector field of $(g, J)$ \cite[p. 1070]{TV07}.
		It follows from the above three inequalities that
		\[
		\sum_{\beta=1}^{2n} u^{i\bar{i}}\frac{|V_{\beta}V_{1}e_{i}(u)|}{\lambda_{1}}\leq
		\ve u^{i\bar{i}}\frac{|e_{i}(\lambda_{1})|^{2}}{\lambda_{1}^{2}}+
		\ve\sum_{\beta>1} u^{i\bar{i}}\frac{|e_{i}(u_{V_{\beta}V_{1}})|^{2}}{\lambda_{1}(\lambda_{1}-\lambda_{\beta})}
		+\frac{C}{\ve}\mathcal{U}.
		\]
		Then, by (\ref{three dv}), we obtain (\ref{4.10}).
	\end{proof}
	Consequently, Proposition \ref{Lma4.2} follows from (\ref{4.10-1})-(\ref{4.10}). 
\end{proof}
\subsection{Proof of Theorem \ref{Thm4.1}} We divide the proof  into three cases.

~\\
\textbf{Case 1.} At $x_{0}$, 
\begin{equation}\label{4.12}
	u^{n\bar{n}}\leq B^{3}e^{2B\varpi}u^{1\bar{1}}.
\end{equation}
\textbf{Case 2.}  At $x_{0}$,  \begin{equation}\label{5.15}
	\frac{\phi'}{4}\sum_{j=1}^{n} u^{i\bar{i}}(|e_{i}e_{j}u|^{2}+|e_{i}\bar{e}_{j}u|^{2})>6\sup_{\bar{\Omega}}(|\partial\varpi|^{2})B^{2}e^{2B\varpi}\mathcal{U}.\end{equation}
In both cases, we choose $\ve=\frac{1}{2}$.  Using  $|a+b|^{2}\leq 4|a|^{2}+\frac{4}{3}|b|^{2}$ for (\ref{5..10}), 
\begin{equation*}
	-(1+\ve)u^{i\bar{i}}\frac{|e_{i}(\lambda_{1})|^{2}}{\lambda_{1}^{2}}\geq
	-6\sup_{\bar{\Omega}}(|\partial \varpi|^{2})B^{2}e^{2B\varpi}\mathcal{U}-2(\phi')^{2} u^{i\bar{i}}|e_{i}(|\partial u|^{2})|^{2}.
\end{equation*}
Substituting this into (\ref{4.7"}),
\begin{equation*}
	\begin{split}
		0 \geq & (2-\ve)\sum_{\alpha>1}u^{i\bar{i}}\frac{|e_{i}(u_{V_{\alpha}V_{1}})|^{2}}{\lambda_{1}(\lambda_{1}-\lambda_{\alpha})}
		+\frac{1}{\lambda_{1}}u^{i\bar{i}}u^{j\bar{j}}|V_{1}(u_{i\bar{j}})|^{2}
		\\&-\Big(\frac{C}{\ve}+6\sup_{\bar{\Omega}}(|\partial \varpi|^{2})B^{2}e^{2B\varpi}\Big)\mathcal{U}
		+ \frac{\phi'}{2}\sum_{j=1}^{n} u^{i\bar{i}}(|e_{i}e_{j}u|^{2}+|e_{i}\bar{e}_{j}u|^{2})\\&
		+Be^{B\varpi}L(\varpi)+B^{2}e^{B\varpi} u^{i\bar{i}}|\varpi_{i}|^{2}-C.\\
	\end{split}
\end{equation*}
~\\
\textbf{Proof of Case 1} Since $L(\varpi)$ is uniformly bounded from below, it follows from the concavity of $L$ that
\begin{equation}\label{z}
	0\geq \frac{\phi'}{2}\sum_{j=1}^{n} u^{i\bar{i}}(|e_{i}e_{j}u|^{2}+|e_{i}\bar{e}_{j}u|^{2})-C_{B}\mathcal{U}.
\end{equation}\footnote{In what follows, $C_{B}$ are positive constants depending on $B$.}
Notice that $\{u^{i\bar{i}}\}$ are pairwisely comparable, by (\ref{4.12}), so
\[
\sum_{i,j}(|e_{i}e_{j}u|^{2}+|e_{i}\bar{e}_{j}u|^{2})\leq C_{B}K.
\]
Thus the complex covariant derivatives
\[
u_{ij}=e_{i}e_{j}u-(\hat{\nabla}_{e_{i}}e_{j})u; \ u_{i\bar{j}}=e_{i}\bar{e}_{j}u-(\hat{\nabla}_{e_{i}}\bar{e}_{j})u
\]
satisfy
\[
\sum_{i,j}(|u_{ij}|^{2}+|u_{i\bar{j}}|^{2})\leq C_{B}K,
\]
and this proves (\ref{4.1}).
\qed
~\\
\textbf{Proof of Case 2} It follows from (\ref{4.7"}) and (\ref{5.15}) that
\begin{equation}\label{4.19}
	\begin{split}
		0 \geq \frac{\phi'}{4}\sum_{j=1}^{n}u^{i\bar{i}}(|e_{i}e_{j}u|^{2}+|e_{i}\bar{e}_{j}u|^{2})
		-\frac{C}{\ve}\mathcal{U}
		+Be^{B\varpi}L(\varpi).
	\end{split}
\end{equation}
Using the fact that
$L(\varpi)\geq  \theta(1+\mathcal{U})$ (by (\ref{2.2})), we have that
\[
0  \geq \frac{\phi'}{4}\sum_{j=1}^{n}u^{i\bar{i}}(|e_{i}e_{j}u|^{2}+|e_{i}\bar{e}_{j}u|^{2})+\Big(\frac{1}{2}\theta Be^{B\varpi}-\frac{C}{\ve}\Big)\mathcal{U}+\frac{1}{2}\theta Be^{B\varpi},
\]
which yields a contradiction if we further assume that $B$ is large enough. \qed

~\\
\textbf{Case 3.} If the Cases 1 and 2 do not hold, we define 
\[
I:=\Big\{1\leq i\leq n:~ u^{n\bar{n}}(x_0)\geq B^{3}e^{2B\varpi(0)}u^{i\bar{i}}(x_0)\Big\}.
\]~\\
Clearly,  $1\in I, n\not\in I$. Hence, we may let $I=\{1,2,\cdots, p\}$ for a certain $p<n$. 
\begin{lemma}
	Assume that $B\geq 6n\sup_{\bar{\Omega}}(|\partial \varpi|^{2})$. At $x_{0}$, we have
	\begin{equation}
		-(1+\ve)\sum_{i\in I}u^{i\bar{i}} \frac{|e_{i}(\lambda_{1})|^{2}}{\lambda_{1}^{2}}\geq -\mathcal{U}-2(\phi')^{2}\sum_{i\in I}u^{i\bar{i}}|e_{i}(|\partial u|^{2})|.
	\end{equation}	
\end{lemma}
\begin{proof}
	It follows from (\ref{5..10}) and the inequality $|a+b|^{2}\leq 4|a|^{2}+\frac{4}{3}|b|^{2}$ that
	\[
	\begin{split}
		&-(1+\ve)\sum_{i\in I} u^{i\bar{i}}\frac{|e_{i}(\lambda_{1})|^{2}}{\lambda_{1}^{2}}\\
		=&-\frac{3}{2}\sum_{i\in I} u^{i\bar{i}}|\phi'e_{i}(|\partial u|^{2})+Ae^{A\varpi}\varpi_{i}|^{2}\\
		\geq & -6\sup_{\bar{\Omega}}(|\partial \varpi|^{2})B^{2}e^{2B\varpi}\sum_{i\in I} u^{i\bar{i}}-2(\phi')^{2}\sum_{i\in I} u^{i\bar{i}}|e_{i}(|\partial u|^{2})|^{2}\\
		\geq & -6n\sup_{\bar{\Omega}}(|\partial \varpi|^{2})B^{-1} u^{n\bar{n}}-2(\phi')^{2}\sum_{i\in I} u^{i\bar{i}}|e_{i}(|\partial u|^{2})|^{2}\\
		\geq &-\mathcal{U}-2(\phi')^{2}\sum_{i\in I}u^{i\bar{i}}|e_{i}(|\partial u|^{2})|^{2},
	\end{split}
	\]
	where we used  $B\geq 6n\sup_{\bar{\Omega}}(|\partial \varpi|^{2})$ in the last inequality.
\end{proof}
Let us define a new (1,0) vector field by
\[
\tilde{e}_{1}:=\frac{1}{\sqrt{2}}(V_{1}-\sqrt{-1}JV_{1}).
\]
At $x_{0}$, there exist $\varsigma_{1},\cdots, \varsigma_{n}\in \mathbb{C}$ such that
\[
\tilde{e}_{1}=\sum_{k=1}^{n}\varsigma_{k}e_{k},\qquad \sum_{k=1}^{n}|\varsigma_{k}|^{2}=1.
\]
\begin{lemma}\label{coeff} At $x_{0}$,
	$
	|\varsigma_{k}|\leq \frac{C_{B}}{\lambda_{1}} ~\textrm{for all} ~k\not\in I.
	$
\end{lemma}
\begin{proof}
	The proof is from \cite{CTW}; we include it here for the convenience of the reader. Now  we have
	\[
	\frac{\phi'}{4}\sum_{i\not\in I}\sum_{j=1}^{n} u^{i\bar{i}}(|e_{i}e_{j}u|^{2}+|e_{i}\bar{e}_{j}u|^{2})\leq 6n^{2}\sup_{\bar{\Omega}}(|\partial \varpi|^{2})B^{2}e^{2B\varpi}u^{n\bar{n}}.
	\]
	When $ u^{n\bar{n}}\leq B^{3}e^{2B\varpi}u^{i\bar{i}}$ for each $i\not\in I$,
	it follows that
	\[
	\sum_{\alpha=2p+1}^{2n}\sum_{\beta=1}^{2n}|\hat{\nabla}^{2}_{\alpha\beta}u|\leq C_{B},
	\]
	which in turn implies that $|\Phi_{\beta}^{\alpha}|\leq C_{B}$ for $2p+1\leq\alpha\leq 2n$, $1\leq \beta\leq 2n$. Since $\Phi(V_{1})=\lambda_{1}V_{1}$, 
	\[
	|V_{1}^{\alpha}|=|\frac{1}{\lambda_{1}}(\Phi(V_{1}))^{\alpha}|=\frac{1}{\lambda_{1}}|
	\sum_{\beta=1}^{2n}\Phi_{\beta}^{\alpha}V_{1}^{\beta}|\leq \frac{C_{B}}{\lambda_{1}}.
	\]
	This proves the lemma.
\end{proof}
Now we estimate the first three terms in Proposition \ref{Lma4.2}. Since $JV_{1}$ is $g$-unitary and  $g$-orthogonal to $V_{1}$, there exist $\mu_{2},\cdots,\mu_{2n}\in \mathbb{R}$ such that
\[
JV_{1}=\sum_{\alpha>1}\mu_{\alpha}V_{\alpha}, \qquad \sum_{\alpha>1}\mu_{\alpha}^{2}=1~ \textrm{at} \ x_{0}.
\]
\begin{lemma}\label{concav}
	At $x_{0}$, for any constant $\gamma>0$, 
	\[
	\begin{split}
		& (2-\ve)\sum_{\alpha>1}u^{i\bar{i}}\frac{|e_{i}(u_{V_{\alpha}V_{1}})|^{2}}{\lambda_{1}(\lambda_{1}-\lambda_{\alpha})}
		+\frac{1}{\lambda_{1}} u^{i\bar{i}}u^{k\bar{k}}|V_{1}(u_{i\bar{k}})|^{2}
		-(1+\ve)\sum_{i\not\in I}u^{i\bar{i}} \frac{|e_{i}(\lambda_{1})|^{2}}{\lambda_{1}^{2}}\\
		\geq &(2-\ve
		)\sum_{i\not\in I}\sum_{\alpha>1}u^{i\bar{i}} \frac{|e_{i}(u_{V_{\alpha}V_{1}})|^{2}}{\lambda_{1}(\lambda_{1}-\lambda_{\alpha})}
		+2\sum_{k\in I,i\not\in I} u^{i\bar{i}}u^{k\bar{k}}\frac{|V_{1}(u_{i\bar{k}})|^{2}}{\lambda_{1}}\\
		&-3\ve\sum_{i\not\in I}u^{i\bar{i}}\frac{|e_{i}(\lambda_{1})|^{2}}{\lambda_{1}^{2}}-
		2(1-\ve)(1+{\gamma})u_{\tilde{1}\bar{\tilde{1}}}\sum_{k\in I,i\not\in I}u^{i\bar{i}}u^{k\bar{k}}\frac{|V_{1}(u_{i\bar{k}})|^{2}}{\lambda_{1}^{2}}\\
		&-\frac{C}{\ve}\mathcal{U}-(1-\ve)(1+\frac{1}{\gamma})(\lambda_{1}-\sum_{\alpha>1}\lambda_{\alpha}
		\mu_{\alpha}^{2})
		\sum_{i\not\in I}\sum_{\alpha>1}u^{i\bar{i}}
		\frac{|e_{i}(u_{V_{\alpha}V_{1}})|^{2}}{\lambda_{1}^{2}(\lambda_{1}-\lambda_{\alpha})},\\
	\end{split}
	\]
	if we assume that $\lambda_{1}\geq \frac{C_{B}}{\ve}$, where $u_{\tilde{1}\bar{\tilde{1}}}:=\sum_{i=1}^{n} u_{i\bar{i}}|\varsigma_{i}|^{2}$.
\end{lemma}
\begin{proof} We divide the proof into three steps.
	
	\textrm{Step 1.}  Since $\bar{\tilde{e}}_{1}=\frac{1}{\sqrt{2}}(V_{1}+\sqrt{-1}JV_{1})$,
	\[
	e_{i}(u_{V_{1}V_{1}})=\sqrt{2} e_{i}(u_{V_{1}\bar{\tilde{e}}_{1}})-\sqrt{-1} e_{i}(u_{V_{1}JV_{1}}).
	\]
We have the first term is
	\begin{equation*}\label{4.27}
		\begin{split}
			e_{i}(u_{V_{1}\bar{\bar{e}}_{1}})= & e_{i}(V_{1}\bar{\tilde{e}}_{1}u-(\hat{\nabla}_{V_{1}}\bar{\bar{e}}_{1})u)
			=\bar{\bar{e}}_{1}e_{i}V_{1}u+O(\lambda_{1})   \\
			= & \sum_{k}\overline{\varsigma_{k}}V_{1}(u_{i\bar{k}})+O(\lambda_{1}),
		\end{split}
	\end{equation*}
	where $O(\lambda_{1})$ are those terms which can be controlled by $\lambda_{1}$. The second term is
	\begin{equation*}\label{4.28}
		\begin{split}
			e_{i}(u_{V_{1}JV_{1}})=&e_{i}{V_{1}JV_{1}}(u)+O(\lambda_{1})=JV_{1}e_{i}{V_{1}}(u)+O(\lambda_{1})  \\
			=&\sum_{\alpha>1}V_{\alpha}e_{i}{V_{1}}(u)+O(\lambda_{1})=\sum_{\alpha>1}e_{i}(u_{V_{\alpha}{V_{1}}})+O(\lambda_{1}).\\
		\end{split}
	\end{equation*}
	Thus, 
	\begin{equation}\label{4.26}
		e_{i}(\lambda_{1})=\sqrt{2}\sum_{k} \overline{\varsigma_{k}}V_{1}(u_{i\bar{k}})-\sqrt{-1}\sum_{\alpha>1}\mu_{\alpha}e_{i}(u_{V_{1}
			V_{\alpha}})+O(\lambda_{1}).\end{equation}
	
	Step 2. It follows from (\ref{4.26}) and Lemma \ref{coeff} that
	\begin{equation}\label{4.26coeff}
		\begin{split}
			&-(1+\ve)\sum_{i\not\in I} u^{i\bar{i}}\frac{|e_{i}(\lambda_{1})|^{2}}{\lambda_{1}^{2}}\\
			\geq& 
			-(1-\ve)\sum_{i\not\in I} u^{i\bar{i}}\frac{|\sqrt{2}\sum_{k\in I} \overline{\varsigma_{k}}V_{1}(u_{i\bar{k}})-\sqrt{-1}\sum_{\alpha>1}\mu_{\alpha}e_{i}(u_{V_{1}
					V_{\alpha}})|^{2}}{\lambda_{1}^{2}}\\
			&-3\ve \sum_{i\not\in I} u^{i\bar{i}}\frac{|e_{i}(\lambda_{1})|^{2}}{\lambda_{1}^{2}}-\frac{C_{B}}{\ve}\sum_{i\not\in I,k\not\in I}u^{i\bar{i}}\frac{|V_{1}(u_{i\bar{k}})|^{2}}{\lambda_{1}^{4}}-\frac{C}{\ve}\mathcal{U}.
		\end{split}
	\end{equation}
	By the Cauchy-Schwarz inequality,
	\begin{alignat}{2} 
		& \Big|\sum_{\alpha>1}\mu_{\alpha}e_{i}(u_{V_{1}V_{\alpha}})\Big|^{2}
		& & \leq \sum_{\alpha>1}(\lambda_{1}-\lambda_{\alpha}\mu_{\alpha}^{2})
		\sum_{\beta>1}\frac{|e_{i}(u_{V_{1}V_{\beta}})|^{2}}{\lambda_{1}-\lambda_{\beta}};\\
		\label{4..25}&\Big|\sum_{k\in I}\overline{\varsigma_{k}}V_{1}(u_{i\bar{k}})\Big|^{2}
		& &\leq u_{\tilde{1}\bar{\tilde{1}}}\sum_{k\in I}u^{k\bar{k}}|V_{1}(u_{i\bar{k}})|^{2}.
	\end{alignat}
	With these, for each $\gamma>0$, 
	\begin{equation}\label{with these}
		\begin{split}
			& (1-\ve)\sum_{i\not\in I} u^{i\bar{i}}\frac{|\sqrt{2}\sum_{k\in I} \overline{\varsigma_{k}}V_{1}(u_{i\bar{k}})-\sqrt{-1}\sum_{\alpha>1}\mu_{\alpha}e_{i}(u_{V_{1}
					V_{\alpha}})|^{2}}{\lambda_{1}^{2}}\\
			\leq & 2(1-\ve)(1+\gamma)\sum_{i\not\in I} u^{i\bar{i}}\frac{|\sum_{k\in I} \overline{\varsigma_{k}}V_{1}(u_{i\bar{k}})|^{2}}{\lambda_{1}^{2}}\\
			&+(1-\ve)(1+\frac{1}{\gamma})\sum_{i\not\in I} u^{i\bar{i}}\frac{|\sum_{\alpha>1}\mu_{\alpha}e_{i}(u_{V_{1}
					V_{\alpha}})|^{2}}{\lambda_{1}^{2}}\\
			\leq & 2(1-\ve)(1+\gamma)u_{\tilde{1}\bar{\tilde{1}}}\sum_{i\not\in I,k\in I}u^{i\bar{i}} u^{k\bar{k}}\frac{|V_{1}(u_{i\bar{k}})|^{2}}{\lambda_{1}^{2}}\\
			&+(1-\ve)(1+\frac{1}{\gamma})(\lambda_{1}-\sum_{\alpha>1} \lambda_{\alpha}\mu_{\alpha}^{2})\sum_{i\not\in I}\sum_{\alpha>1} u^{i\bar{i}}\frac{|e_{i}(u_{V_{\alpha}V_{1}})|^{2}}{\lambda_{1}^{2}(\lambda_{1}-\lambda_{\alpha})}.
		\end{split}
	\end{equation}
	
	Step 3.
	If $\lambda_{1}\geq \frac{C_{B}}{\ve}$ (by assumption), we know that $u_{1\bar{1}}$ is comparable to $\lambda_{1}$, whence
	$\frac{C_{B}}{\ve\lambda_{1}^{3}}\leq u^{1\bar{1}}\leq u^{k\bar{k}}$
	for all $k$. Thus,
	\begin{equation}\label{5..25}
		\begin{split}
			u^{i\bar{i}}u^{k\bar{k}}|V_{1}(u_{i\bar{k}})|^{2}&\geq 2\sum_{k\in I,i\not\in I}u^{i\bar{i}}u^{k\bar{k}}|V_{1}(u_{i\bar{k}})|^{2}+\frac{C_{B}}{\ve}\sum_{i,k\not\in I}u^{i\bar{i}}\frac{|V_{1}(u_{i\bar{k}})|^{2}}{\lambda_{1}^{3}}.
		\end{split}
	\end{equation}
	Then the lemma follows from (\ref{4.26coeff}), (\ref{with these}) and (\ref{5..25}).
\end{proof}
\begin{lemma}\label{final}
	At $x_{0}$, if $\lambda_{1}\geq \frac{C}{\ve^{3}}$,
	\[
	\begin{split}
		&(2-\ve)\sum_{\alpha>1}u^{i\bar{i}}\frac{|e_{i}(u_{V_{\alpha}V_{1}})|^{2}}
		{\lambda_{1}(\lambda_{1}-\lambda_{\alpha})}+\frac{1}{\lambda_{1}} u^{i\bar{i}}u^{k\bar{k}}|V_{1}(u_{i\bar{k}})|^{2}
		-(1+\ve)\sum_{i\not\in I} u^{i\bar{i}}\frac{|e_{i}(\lambda_{1})|^{2}}{\lambda_{1}^{2}}\\
		\geq& -6\ve B^{2}e^{2B\varpi}\sum_{i=1}^{n} u^{i\bar{i}}|\varpi_{i}|^{2}-6\ve (\phi')^{2}\sum_{i\not\in I}u^{i\bar{i}}|e_{i}(|\partial u|^{2})|^{2}-\frac{C}{\ve}\mathcal{U}.
	\end{split}
	\]
\end{lemma}
\begin{proof} It suffices to prove that
\begin{equation}\label{4.30}
	\begin{split}
		&(2-\ve)\sum_{\alpha>1}u^{i\bar{i}}\frac{|e_{i}(u_{V_{\alpha}V_{1}})|^{2}}
		{\lambda_{1}(\lambda_{1}-\lambda_{\alpha})}+\frac{1}{\lambda_{1}} u^{i\bar{i}}u^{k\bar{k}}|V_{1}(u_{i\bar{k}})|^{2}\\&
		-(1+\ve)\sum_{i\not\in I} u^{i\bar{i}}\frac{|e_{i}(\lambda_{1})|^{2}}{\lambda_{1}^{2}}
		\geq -3\ve \sum_{i\not\in I} u^{i\bar{i}}\frac{|e_{i}(\lambda_{1})|^{2}}{\lambda_{1}^{2}}-\frac{C}{\ve} \mathcal{U}.
	\end{split}
\end{equation}
We divide the proof into two assumptions.

\textbf{Assumption 1:} At $x_{0}$, we assume that
\begin{equation}\label{5.32}
	\lambda_{1}+\sum_{\alpha>1}\lambda_{\alpha}\mu_{\alpha}^{2}\geq 2(1-\ve)u_{\tilde{1}\bar{\tilde{1}}}>0.
\end{equation}
\begin{proof} Taking  this, as well as Lemma \ref{concav}, we get that
	\begin{equation}\label{5.76}
		\begin{split}
			& (2-\ve)\sum_{\alpha>1}u^{i\bar{i}}\frac{|e_{i}(u_{V_{\alpha}V_{1}})|^{2}}{\lambda_{1}(\lambda_{1}-\lambda_{\alpha})}
			+\frac{1}{\lambda_{1}} u^{i\bar{i}}u^{k\bar{k}}|V_{1}(u_{i\bar{k}})|^{2} -(1+\ve)\sum_{i\not\in I} u^{i\bar{i}}\frac{|e_{i}(\lambda_{1})|^{2}}{\lambda_{1}^{2}}\\
			\geq &(2-\ve)\sum_{i\not\in I}\sum_{\alpha>1}u^{i\bar{i}}\frac{|e_{i}(u_{V_{\alpha}V_{1}})|^{2}}{\lambda_{1}(\lambda_{1}-\lambda_{\alpha})}
			+\sum_{k\in I,i\not\in I}\frac{2}{\lambda_{1}} u^{i\bar{i}}u^{k\bar{k}}|V_{1}(u_{i\bar{k}})|^{2}\\
			&-3\ve\sum_{i\not\in I}u^{i\bar{i}}\frac{|e_{i}(\lambda_{1})|^{2}}{\lambda_{1}^{2}}-
			(1+{\gamma})(\lambda_{1}+\sum_{\alpha>1}\lambda_{\alpha}\mu_{\alpha}^{2})\sum_{k\in I,i\not\in I}u^{i\bar{i}}u^{k\bar{k}}\frac{|V_{1}(u_{i\bar{k}})|^{2}}{\lambda_{1}^{2}}\\
			&-\frac{C}{\ve}\mathcal{U}-
			(1-\ve)(1+\frac{1}{\gamma})(\lambda_{1}-\sum_{\alpha>1}\lambda_{\alpha}\mu_{\alpha}^{2})
			\sum_{i\not\in I}\sum_{\alpha>1}u^{i\bar{i}}\frac{|e_{i}(u_{V_{\alpha}V_{1}})|^{2}}{\lambda_{1}^{2}(\lambda_{1}-\lambda_{\alpha})}.
		\end{split}
	\end{equation}
	We only choose that
	$\gamma=\frac{\lambda_{1}-\underset{\alpha>1}{\sum}\lambda_{\alpha}\mu_{\alpha}^{2}}
	{\lambda_{1}+\underset{\alpha>1}{\sum}\lambda_{\alpha}\mu_{\alpha}^{2}}.$
	On the right side of (\ref{5.76}), the first term cancels the last term, and the second term cancels the fourth. This proves  (\ref{4.30}). 
\end{proof}
~\\
\textbf{Assumption 2:}  At $x_{0}$, we assume that
\begin{equation}\label{4.31-1}
	{\lambda_{1}+\sum_{\alpha>1}\lambda_{\alpha}\mu_{\alpha}^{2}}< 2(1-\ve)u_{\tilde{1}\bar{\tilde{1}}}.
\end{equation} 
\begin{proof}
	Computing at $x_{0}$, we get that
	\[
	\begin{split}
		u_{\tilde{1}\bar{\tilde{1}}}=&(\ddbar u)(\widetilde{e}_{1}, \overline{\widetilde{e}_{1}})= \sum_{i=1}^{n} \big\{e_{i}\bar{e}_{i}(u)-[e_{i},\bar{e}_{i}]^{(0,1)}(u)\big\}|\varsigma_{i}|^{2}
		\\
		\leq & \frac{1}{2}\Big\{V_{1}V_{1}(u)+(JV_{1})(JV_{1})(u)+\sqrt{-1}[V_{1},JV_{1}](u)\Big\}- [\widetilde{e}_{1}, \overline{\widetilde{e}_{1}}]^{(0,1)}(u)+C\\
		\leq & \frac{1}{2}\Big\{u_{V_{1}V_{1}}+u_{JV_{1}JV_{1}}+(\hat{\nabla}_{V_{1}}V_{1})(u)+(\hat{\nabla}_{JV_{1}}JV_{1})(u)+\sqrt{-1}[V_{1},JV_{1}](u)\Big\}+C\\
		\leq& \frac{1}{2}(\lambda_{1}+\sum_{\alpha>1}\lambda_{\alpha}\mu_{\alpha}^{2})+C.
	\end{split}
	\]
	It then follows from (\ref{4.31-1}) that
	$
	\lambda_{1}+\sum_{\alpha>1} \lambda_{\alpha}\mu_{\alpha}^{2}\geq -C
	$
	and $u_{\tilde{1}\bar{\tilde{1}}}\leq \frac{C}{\ve}$. Hence,
	$0<\lambda_{1}-\sum_{\alpha>1} \lambda_{\alpha}\mu_{\alpha}^{2}\leq 2\lambda_{1}+C\leq (2+2\ve^{2})\lambda_{1}$,
	provided that $\lambda_{1}\geq \frac{C}{\ve^{2}}$. Choosing $\gamma=\frac{1}{\ve^{2}}$, 
	\[
	\begin{split}
		(1-\ve)(1+\frac{1}{\gamma})(\lambda_{1}-\sum_{\alpha>1}\lambda_{\alpha}\mu_{\alpha}^{2})
		\leq & 2(1-\ve)(1+\ve^{2})^{2}\lambda_{1}
		\leq (2-\ve)\lambda_{1}.
	\end{split}
	\]
	Substituting this into Lemma \ref{concav} yields that
	\[
	\begin{split}
		& (2-\ve)\sum_{\alpha>1}u^{i\bar{i}}\frac{|e_{i}(u_{V_{\alpha}V_{1}})|^{2}}{\lambda_{1}(\lambda_{1}-\lambda_{\alpha})}
		+\frac{1}{\lambda_{1}} u^{i\bar{i}}u^{k\bar{k}}|V_{1}(u_{i\bar{k}})|^{2} -(1+\ve)\sum_{i\not\in I} u^{i\bar{i}}\frac{|e_{i}(\lambda_{1})|^{2}}{\lambda_{1}^{2}}\\
		\geq &2\sum_{k\in I,i\not\in I} u^{i\bar{i}}u^{k\bar{k}}\frac{|V_{1}(u_{i\bar{k}})|^{2}}{\lambda_{1}}-3\ve\sum_{i\not\in I}u^{i\bar{i}}\frac{|e_{i}(\lambda_{1})|^{2}}{\lambda_{1}^{2}}\\
		&-2(1-\ve)(1+\frac{1}{\ve^{2}})u_{\tilde{1}\bar{\tilde{1}}}\sum_{k\in I,i\not\in I}u^{i\bar{i}}u^{k\bar{k}}\frac{|V_{1}(u_{i\bar{k}})|^{2}}{\lambda_{1}^{2}}    -\frac{C}{\ve}\mathcal{U}\\
			\end{split}
	\]
	\[
	\begin{split}
		\geq &2\sum_{k\in I,i\not\in I} u^{i\bar{i}}u^{k\bar{k}}\frac{|V_{1}(u_{i\bar{k}})|^{2}}{\lambda_{1}} -3\ve\sum_{i\not\in I}u^{i\bar{i}}\frac{|e_{i}(\lambda_{1})|^{2}}{\lambda_{1}^{2}}\\
		&-(1-\ve)(1+\frac{1}{\ve^{2}})\frac{C}{\ve}\sum_{k\in I,i\not\in I}u^{i\bar{i}}u^{k\bar{k}}\frac{|V_{1}(u_{i\bar{k}})|^{2}}{\lambda_{1}^{2}}    -\frac{C}{\ve}\mathcal{U}\\
		\geq &-3\ve\sum_{i\not\in I}u^{i\bar{i}}\frac{|e_{i}(\lambda_{1})|^{2}}{\lambda_{1}^{2}}-\frac{C}{\ve}\mathcal{U},\\
	\end{split}
	\]
	where in the last inequality we relied on the fact that $\lambda_{1}\geq \frac{C}{\ve^{3}}$ . This proves  (\ref{4.30}), and hence the proof of the lemma is complete.
\end{proof}
Now  we complete the proof of the interior second order estimate. It follows from Lemma \ref{final} and (\ref{4.7"}) that, at $x_{0}$,
\begin{equation}\label{}
	\begin{split}
		0 \geq & -6\ve B^{2}e^{2B\varpi} u^{i\bar{i}}|\varpi_{i}|^{2}-6\ve (\phi')^{2}\sum_{i\not\in I}u^{i\bar{i}}|e_{i}(|\partial u|^{2})|^{2}-\frac{C}{\ve}\mathcal{U}\\
		&+ \frac{\phi'}{2}\sum_{j=1}^{n} u^{i\bar{i}}(|e_{i}e_{j}u|^{2}+|e_{i}\bar{e}_{j}u|^{2})
		+B^{2}e^{B\varpi} u^{i\bar{i}}|\varpi_{i}|^{2}+Be^{B\varpi}L(\varpi) +\phi''u^{i\bar{i}}|e_{i}(|\partial u|^{2})|^{2}.
	\end{split}
\end{equation}
Choosing $\ve< \frac{1}{6}$ such that $6\ve e^{B\varpi(x_0)}=1$, and by $\phi''=2(\phi')^{2}$,
\[
\begin{split}
	0\geq &-\frac{C}{\ve}\mathcal{U}+ \frac{\phi'}{2}\sum_{j=1}^{n} u^{i\bar{i}}(|e_{i}e_{j}u|^{2}+|e_{i}\bar{e}_{j}u|^{2})+Be^{B\varpi}L(\varpi).
\end{split}
\]
Thus,
$$B\theta e^{B\varpi} +(B\theta-C)e^{B\varpi}\mathcal{U}+ \frac{\phi'}{2}\sum_{j=1}^{n} u^{i\bar{i}}(|e_{i}e_{j}u|^{2}+|e_{i}\bar{e}_{j}u|^{2})\leq 0.$$
We choose $B$ sufficiently large such that $B\theta\geq C$. This then yields a contradiction, and we have completed the proof.\end{proof}
\end{proof}
\begin{remark}
	The  interior $C^{2,\alpha}$ estimates follow from the Evans-Krylov theorem and an extension trick  introduced by Wang \cite{Wan12} in the study of the complex Monge-Amp\`{e}re equation. Then the higher order estimates can be obtained by Schauder estimates.
\end{remark}
\section{Boundary $C^{2}$ estimates}
In this section we shall derive the  estimate
$$\underset{\partial{\Omega}}{\max}|\ddbar u|\leq C$$
for a certain dependent constant $C$.
\subsection{Pure tangential estimates}
Let us fix a point $z\in \partial\Omega$, and define 
$$\rho(x):=\text{dist}_{g}(x,z) \qquad \textrm{in} \ M.$$ 
Since $u-\underline{u}=0$ on $\partial{\Omega}$,  we can write $u=\underline{u}+\rho\sigma$ in a neighborhood of $z$, where $\sigma$ is a function defined  on $\partial\Omega$ which  depends, linearly on the first order derivatives of $u-\underline{u}$. For arbitrary vector fields $X,Y$ which are tangential to $\partial\Omega$, 
$$XY(u)=XY(\underline{u})+XY(\rho)\cdot\sigma.$$
It follows from the $C^{1}$ estimate that
\begin{equation}\label{pure `tangent}
	|XY(u)|(z)\leq C.	
\end{equation} 
Then the pure  tangential estimates follow by the randomicity of $z$.

\subsection{Mixed direction estimates}
\begin{proposition}\label{Mixed direction estimate.}
	Let $N\in T_{z}M$ be orthogonal to $\partial\Omega$ such that $N\rho=-1$, and let $X$ be a vector field  which is tangential to $\partial\Omega$. We have that
	\begin{equation}\label{mix dir}
		|NX(u)|(z)\leq C,
	\end{equation}
	where $C$ depends on $\|u\|_{C^{1}(\bar{\Omega})}$, $h$, $\|\underline{u}\|_{C^{2}}$ and other known data.
\end{proposition}
\begin{proof}

	Let $\mathcal{O}\subseteq M$ be a local coordinate chart with $z\in \mathcal{O}$. We may pick up real vector fields $X_{1},\cdots, X_{n}$ which are tangential at $z$ to $\partial\Omega$  such that $X_{1}, JX_{1},\cdots, X_{n}, JX_{n}$ is a $g$-orthonormal local frame near $z$. Furthermore, we assume  that
	$Y_{n}:=JX_{n}$ is the normal vector on $\partial\Omega$ near $z$.
	
	Fixing a constant $\delta>0$, we set 
	$$\Omega_{\delta}:=\{x\in \Omega \mid \rho(x)\leq \delta\}.$$ 
	Notice that $\sqrt{-1}\partial\bar{\partial}\rho^{2}=\omega$ at $z$.
	By continuity, we may rearrange $\delta\ll 1$ such that 
	\begin{equation*}
		\frac{1}{2}\omega\leq \sqrt{-1}\partial\bar{\partial}\rho^{2}\leq 2\omega \qquad \text{in} \ \Omega_{\delta}.
	\end{equation*}
	We shall prove (\ref{mix dir}) by applying the maximum principle to
	\begin{equation*}
		\begin{split}
			Q_{\pm}&=\pm X(u-\underline{u})+\sum_{j=1}^{ n}|X_{j}(u-\underline{u})|^{2}+Av-B\rho^{2}	\\
		\end{split}
	\end{equation*}
	for a negative function $v\in C^{\infty}(\Omega_{\delta})$ to be determined later. Let $\mathcal{O}'\subsetneq \mathcal{O}$  be a neighborhood of $z$, and set $S_{\delta}:=\mathcal{O}'\cap \Omega_{\delta}$. 
	
	First we choose $B$ large enough such that $Q_{\pm}\leq 0$ on $\partial S_{\delta}$. We shall prove $Q_{\pm}\leq 0$ in $\bar{S_{\delta}}$ for a large constant $A$. 
	Otherwise,  suppose that $Q_{\pm}$ attains its maximum at a point $x_{0}\in S_{\delta}$. Let $e_{1},\cdots,e_{n}$ with
	$$e_{i}:=\frac{1}{\sqrt{2}}(X_{i}-\sqrt{-1}JX_{i}), \ 1\leq i\leq n$$
	be a local $g$-orthonormal frame in a neighborhood of $x_{0}$ such that the matrix $(u_{i\bar{j}})$ is diagonal at $x_{0}$.
	
	The following lemma plays a significant role in our proof:
	\begin{lemma}
		There exist some uniform positive constants $t$, $\delta$ and $\varepsilon$ sufficiently small, and an $N$ sufficiently large, such that the function
		\begin{equation}
			v:=\underline{u}-u-td+Nd^{2}
		\end{equation}
		satisfies $v\leq 0$ in $\bar{\Omega}_{\delta}$ and
		\begin{equation}\label{Lv}
			L(v)\geq \varepsilon (1+\mathcal{U}) \qquad \text{at} \ x_{0}.
		\end{equation}
	\end{lemma}

	\begin{proof}
		As $\underline{u}\leq u$ and $v\leq 0$ in $\bar{\Omega}_{\delta}$, if we let $\delta\ll t$ be small enough such that $N\delta<t$, then by a direct calculation and the property of the mixed discriminant, at $x_{0}$, 
		\begin{equation*}
			\begin{split}
				L(\underline{u}-u)\geq& n\tau h^{-1}\text{det}(I,\sqrt{-1}\partial\bar{\partial}u[n-1])-n
				=\tau\mathcal{U}-n;
			\end{split}
		\end{equation*}
		\begin{equation*}
			\begin{split}
				L(-td+Nd^{2})=&-(t-Nd)u^{i\bar{j}}d_{i\bar{j}}+Nu^{i\bar{j}}d_{i}d_{\bar{j}}
				\geq-C_{1}(t-Nd) \mathcal{U}+\frac{N}{2}\min_{1\leq i\leq n}u^{i\bar{i}},
			\end{split}
		\end{equation*}
		where we used \eqref{tau}. It follows that
		\begin{equation}\label{combing}
			\begin{split}
				L(v)\geq& (\tau-C_{1}t)\mathcal{U}+\frac{N}{2}\min_{1\leq i\leq n}u^{i\bar{i}}-n
				\geq \frac{\tau}{2}\mathcal{U}+\frac{N}{2}\min_{1\leq i\leq n}u^{i\bar{i}}-n,
			\end{split}
		\end{equation}
		if $t\ll 1$.	By an elementary inequality, we deduce that
		\begin{equation*}
			\begin{split}
				\frac{\tau}{4}\mathcal{U}+\frac{N}{2}\min_{1\leq i\leq n}u^{i\bar{i}}
				\geq & n\Big(\frac{\tau}{4}\Big)^{\frac{n-1}{n}}\Big(N\prod_{1\leq i\leq n} u^{i\bar{i}}\Big)^{\frac{1}{n}}\\
				\geq & n\Big(\frac{\tau}{4}\Big)^{\frac{n-1}{n}}N^{\frac{1}{n}}h^{-\frac{1}{n}}
				\geq C_{2}\Big(\frac{\tau}{4}\Big)^{\frac{n-1}{n}}
				N^{\frac{1}{n}}.
			\end{split}
		\end{equation*}
		We choose $N$ large enough such that $$C_{2}\Big(\frac{\tau}{4}\Big)^{\frac{n-1}{n}}N^{\frac{1}{n}}\geq \frac{\tau}{4}+n.$$
		Substituting this into (\ref{combing}), we get that $L(v)\geq \frac{\tau}{4}(1+\mathcal{U})$. This completes the proof.
	\end{proof}	
	Now we continue to prove  Proposition \ref{Mixed direction estimate.}. Clearly,
	\begin{equation}\label{first we have}
		L(\mp\underline{u}-B\rho^{2})\geq -BC\mathcal{U}.
	\end{equation}
	For each vector field $Y$, 
	\begin{equation*}
		\begin{split}
			L(Yu)&=u^{i\bar{j}}(e_{i}\bar{e}_{j}Yu-[e_{i},\bar{e}_{j}]^{0,1}Yu)\\
			&=Y(h)+u^{i\bar{j}}\Big(e_{i}[\bar{e}_{j},Y]u
			+[e_{i},Y]\bar{e}_{j}u-\big[[e_{i},\bar{e}_{j}]^{0,1},Y\big]u\Big).
		\end{split}
	\end{equation*}
	There exist $\alpha_{jk}, \beta_{jk}\in \mathbb{C}$ such that
	\begin{equation*}
		[e_{j},Y]=\sum_{k=1}^{n}\alpha_{jk}e_{k}+ \beta_{jk}X_{k}; \ [\bar{e}_{j},Y]=\sum_{k=1}^{n}\overline{\alpha_{jk}}\bar{e}_{k}+ \overline{\beta_{jk}}X_{k}.
	\end{equation*}
	It follows that
	\begin{equation*}
		L(Yu)\leq Cu^{i\bar{i}}\Big(1+\sum_{k=1}^{n}|e_{i}X_{k}u|\Big),
	\end{equation*}
	which implies that
	\begin{equation}\label{mix}
		\begin{split}
			&L(\pm X u+\sum_{j=1}^{n}|X_{j}(u-\underline{u})|^{2})\\	
			\geq &	u^{i\bar{i}}\sum_{j=1}^{n}(e_{i}X_{j}(u-\underline{u}))(\bar{e}_{i}X_{j}(u-\underline{u}))-Cu^{i\bar{i}}(1+\sum_{j=1}^{n}|e_{i}X_{j}u|)\\
			\geq & \frac{1}{2}u^{i\bar{i}}\sum_{j=1}^{n}|e_{i}X_{j}u|^{2}-Cu^{i\bar{i}}(1+\sum_{j=1}^{n}|e_{i}X_{j}u|)
			\geq  -C\mathcal{U},
		\end{split}
	\end{equation}
	where in the last inequality we used the fact that $\frac{1}{2}a^{2}+2ab\geq -2b^{2}$. 
	It then follows from (\ref{Lv}), (\ref{first we have}) and (\ref{mix} ) that 
	\begin{equation*}
		L(Q_{\pm})(x_{0})\geq A\varepsilon+(A\varepsilon-BC-C)\mathcal{U}>0,
	\end{equation*}
	if $A$ is large enough such that $A\varepsilon\geq (B+1)C$, which  contradicts to the fact that $Q_{\pm}$ attains its
	maximum at $x_{0}$. Consequently, $Q_{\pm}\leq 0$  in $\bar{S}_{\delta}$ and $Q_{\pm}(z)=0$. By Hopf's lemma, $|NXu|(z)\leq C$.
\end{proof}
\subsection{Pure normal estimates}
\begin{proposition}\label{double}
	Let $N\in T_{z}M$ be orthogonal to $\partial\Omega$ at $z$ such that $N\rho=-1$. We have 
	\begin{equation}\label{double normal}
		|NN(u)|(z)\leq C,
	\end{equation}
	where $C$ depends on $\|u\|_{C^{1}(\bar{\Omega})}$,  $h$, $\|\underline{u}\|_{C^{2}}$ and other known data.
\end{proposition}

Before proving this, let us recall some useful facts from the matrix theory. For any Hermitian matrix $A=(a_{i\bar{j}})$ with eigenvalues $\lambda_{i}(A)$, let $\tilde{A}:=(a_{\alpha\bar{\beta}})$, and  we denote  the eigenvalues of $\tilde{A}$  by $\lambda'_{\alpha}(\tilde{A})$.\footnote{In what follows, we let
	$\alpha,\beta=1,2,\cdots, n-1;$
	$i,j=1,2,\cdots,n.$	} It follows from Cauchy's interlace inequality \cite{Hwa04} and \cite[p. 272]{CNS85} that when $|a_{n\bar{n}}|\rightarrow \infty$,
\begin{equation}\label{matrix theory}
	\begin{split}
		& 	\lambda_{\alpha}(A)\leq  \lambda'_{\alpha}(\tilde{A})\leq \lambda_{\alpha+1}(A);\\
		&	\lambda_{\alpha}(A)=\lambda'_{\alpha}(\tilde{A})+O(1);\\
		&a_{n\bar{n}}\leq 	\lambda_{n}(A)\leq  a_{n\bar{n}}\Big(1+O\Big(\frac{1}{a_{n\bar{n}}}\Big)\Big).
	\end{split}
\end{equation}
\begin{proof}
	Let $U:=(u_{i\bar{j}})$ (resp. $\underline{U}:=(\underline{u}_{i\bar{j}})$) be the Hessian matrix of $u$ (resp. $\underline{u}$). We assert that there are uniform constants $c_{0}, R_{0}>0$ such that, for all $R\geq R_{0}$, 
	$(\lambda'(\tilde{U}),R)\in \Gamma_{n}$ and 
	\begin{equation*}\label{suffices}
		\log\det(\lambda'(\tilde{U}),R)\geq h+c_{0}, \qquad \text{on} \ \partial\Omega.
	\end{equation*}

	To this end, we  follow an idea of Trudinger \cite{Tru95} and set
	\begin{equation*}\label{trudinger}
		\tilde{m}:=\underset{R\rightarrow \infty}{\lim\inf}\underset{\partial\Omega}{\min}\big(	\log\det\big(\lambda'(\tilde{U}),R\big)-h\big).
	\end{equation*}
	Then we are reduced to showing
	\begin{equation}\label{reduced goal}
		\tilde{m}\geq c_{0}>0.
	\end{equation} 
	
	We may assume that $\tilde{m}<\infty$, otherwise we are done. Supposing that $\tilde{m}$ is attained at a point $x_{0}\in \partial\Omega$,  we pick up a local $g$-orthonormal frame $(e_{1}, \cdots, e_{n})$ as in the previous subsection such that the matrix $(\tilde{U}_{\alpha\bar{\beta}}(x_{0}))$ is diagonal. We choose real vector fields $X_{1},\cdots, X_{n}$ tangential at $x_{0}$ to $\partial\Omega$  such that $X_{1}, JX_{1},\cdots, X_{n}, JX_{n}$ constitute a $g$-orthonormal local frame near $x_{0}$, and
	$Y_{n}:=JX_{n}$ is the normal vector on $\partial\Omega$ near $x_{0}$.  
	Letting	$$\Gamma_{\infty}:=\big\{(\lambda_{1},\cdots,\lambda_{n-1}) \mid \lambda_{\alpha}>0, \ 1\leq \alpha\leq n-1\big\}$$
	be a positive orthant in $\mathbb{R}^{n-1}$, we divide the proof into two cases.
	
	~\\
	\textbf{Case 1.} Assume that it holds that
	\begin{equation}
		\lim_{\lambda_{n}\rightarrow \infty}\sigma_{n}(\lambda', \lambda_{n})=\infty, \qquad \text{for any} \ \lambda'\in \Gamma_{\infty}.
	\end{equation}
	
	By virtue of (\ref{pure `tangent}) and (\ref{mix dir}), we know that
	\begin{equation*}
		\lambda'(\tilde{U})(x_{0})\in \mathcal{C},
	\end{equation*} 
	where $\mathcal{C}\subset  \Gamma_{\infty}$ is compact. Then there  exist $c_{1},R_{1}\in \mathbb{R}_{>0}$ depending on $\lambda'(\tilde{U}(x_{0}))$ such that 
	\begin{equation*}
		\log\det\big(\lambda'(\tilde{U}(x_{0})),R\big)\geq h(x_0)+c_{1}, \qquad \textrm{for any} \ R\geq R_{1}.
	\end{equation*}
	By continuity, there exists a cone $\hat{\mathcal{C}}\subset \Gamma_{\infty}$ and a neighborhood of $\mathcal{C}$ such that
	\begin{equation}\label{case 1}
		\log\det\big(\lambda',R\big)\geq h(x_0)+\frac{c_{1}}{2}, \qquad \textrm{for any} \ \lambda'\in  \hat{\mathcal{C}} \ \textrm{and} \ R\geq R_{1}.
	\end{equation}
	Now we apply (\ref{matrix theory}) to $U=(u_{i\bar{j}})$,  and there exists a large constant $R_{2}\geq R_{1}$ satisfying if $u_{n\bar{n}}(x_{0})\geq R_{2}$, then
	\begin{equation}\label{max eigen}
		\lambda_{n}(U)(x_{0})\geq u_{n\bar{n}}(x_{0})\geq R_{2}\geq R_{1}.
	\end{equation}
We can shrink $\hat{\mathcal{C}}$ if necessary such that
	\begin{equation}\label{other eigen}
		\big(\lambda_{1}(U)(x_{0}),\cdots, \lambda_{n-1}(U)(x_{0})\big)\in \hat{\mathcal{C}}.
	\end{equation}
	It follows from (\ref{case 1}), (\ref{max eigen}) and (\ref{other eigen}) that 
	\begin{equation*}
		\log\det(u_{i\bar{j}})(x_0)\geq h(x_{0})+\frac{c_{1}}{2},
	\end{equation*}
	which yields a contradiction to (\ref{2.1}). Hence (\ref{reduced goal}) follows that by letting $c_{0}:=\frac{c_{1}}{2}$.
	
	~\\
	\textbf{Case 2.} Assume that it holds that
	\begin{equation}
		\lim_{\lambda_{n}\rightarrow \infty}\sigma_{n}(\lambda', \lambda_{n})<\infty, \qquad \text{for any} \ \lambda'\in \Gamma_{\infty}.
	\end{equation}
	
	We define 
	$$\tilde{F}(E):=\lim_{R\rightarrow\infty}\log\det(\lambda'(E),R)$$
	on the set of $(n-1)^{2}$ Hermitian matrices with $\lambda'(E)\in \Gamma_{\infty}$. Notice that $\tilde{F}$ is concave and finite, since the operator $\lambda\mapsto\log\det(\lambda)$ is concave and continuous. Hence,  there exists a symmetric matrix $(\tilde{F}^{\alpha\bar{\beta}})$ such that
	\begin{equation}\label{concave}
		\tilde{F}^{\alpha\bar{\beta}}(\tilde{U})\Big(E_{\alpha\bar{\beta}}-\tilde{U}_{\alpha\bar{\beta}}\Big)\geq \tilde{F}(E)-\tilde{F}(\tilde{U})
	\end{equation}
	for any $(n-1)^{2}$ Hermitian matrix $E$. On $\partial\Omega$, since $u=\underline{u}$,
	\begin{equation*}
		\begin{split}
			\tilde{U}_{\alpha\bar{\beta}}-\tilde{\underline{U}}_{\alpha\bar{\beta}}=&	\nabla_{\bar{\beta}}\nabla_{\alpha}(u-\underline{u})	
			=-g(Y_{n},\nabla_{\alpha}\bar{e}_{\beta})Y_{n}(u-\underline{u}),
		\end{split}
	\end{equation*}
	where $\nabla_{\alpha}\bar{e}_{\beta}=[e_{\alpha},\bar{e}_{\beta}]^{(0,1)}$ (cf. \cite{LZ20}). This, together with (\ref{concave}), yield that
	\begin{equation*}
		\begin{split}
			&Y_{n}(u-\underline{u})(x_{0})	\tilde{F}^{\alpha\bar{\beta}}(\tilde{U}(x_{0}))g(Y_{n},\nabla_{\alpha}\bar{e}_{\beta})\\
			\geq& \tilde{F}(\tilde{\underline{U}}(x_{0}))-\tilde{F}(\tilde{U}(x_{0}))
			= \tilde{F}(\tilde{\underline{U}}(x_{0}))-\tilde{m}-h(x_{0})\\
			\geq &\tilde{F}(\tilde{\underline{U}}(x_{0}))-\log\det(\lambda(\underline{U}))(x_{0})-\tilde{m}
			\geq \tilde{c}-\tilde{m},
		\end{split}
	\end{equation*}
	where 
	$$\tilde{c}:=
	\underset{R\rightarrow \infty}{\lim\inf}\underset{\partial\Omega}{\min}\big[	\log\det(\lambda'(\tilde{\underline{U}}),R)-\log\det(\lambda(\underline{U}))\big].$$ Notice that $0<\tilde{c}<\infty$, since the operator $\lambda\mapsto\log\det(\lambda)$ is strictly increasing with respect to each variable.
	Now we divide the proof into two cases.
	
	\textbf{Subcase 2 (i)} Assume that at $x_{0}$,
	\begin{equation}
		Y_{n}(u-\underline{u})	\tilde{F}^{\alpha\bar{\beta}}(\tilde{U})g(Y_{n},\nabla_{\alpha}\bar{e}_{\beta})\leq \frac{\tilde{c}}{2}.
	\end{equation}
	Given this, $\tilde{m}\geq \frac{\tilde{c}}{2}$, and by choosing $c_{0}=\frac{\tilde{c}}{2}$, we  are done.
	
	\textbf{Subcase 2 (ii)} Assume that at $x_{0}$,
	\begin{equation}\label{case 2}
		Y_{n}(u-\underline{u})	\tilde{F}^{\alpha\bar{\beta}}(\tilde{U})g(Y_{n},\nabla_{\alpha}\bar{e}_{\beta})\geq \frac{\tilde{c}}{2}.
	\end{equation}
	Define $$\eta:=\tilde{F}^{\alpha\bar{\beta}}(\tilde{U}(x_{0}))g(Y_{n},\nabla_{\alpha}\bar{e}_{\beta}) \qquad  \text{on} \ \partial\Omega.$$ 
	Notice that $Y_{n}(u-\underline{u})(x_{0})\geq 0$, and by (\ref{case 2}), is strictly positive. Thus
	\begin{equation*}
		\eta	\geq \frac{\tilde{c}}{2Y_{n}(u-\underline{u})}\geq 2\tau\tilde{c} \qquad \text{at} \ x_{0}
	\end{equation*}
	for some uniform constant $\tau>0$. We may assume that $\eta\geq \tau\tilde{c}$ in $\Omega_{\delta}$ by shrinking $\delta$ again if necessary.
	
	Let us define a function in $\Omega_{\delta}$ by
	\begin{equation*}
		\begin{split}
			\Phi(x)=&\frac{1}{\eta(x)}\tilde{F}^{\alpha\bar{\beta}}(\tilde{U}(x_{0}))	\Big(\tilde{\underline{U}}_{\alpha\bar{\beta}}(x)-\tilde{U}_{\alpha\bar{\beta}}(x_{0})\Big)
			-\frac{h(x)-h(x_{0})}{\eta(x)}-Y_{n}(u-\underline{u})(x)\\
			:=& Q(x)-Y_{n}(u-\underline{u})(x).
		\end{split}
	\end{equation*}
	By a direct calculation,
	\begin{equation*}
		-\eta(x)Y_{n}(u-\underline{u})(x)=\tilde{F}^{\alpha\bar{\beta}}(\tilde{U}(x_{0}))	\Big(\tilde{U}_{\alpha\bar{\beta}}(x)-\tilde{\underline{U}}_{\alpha\bar{\beta}}(x)\Big).	
	\end{equation*}
	It follows from (\ref{concave}) that
	\begin{equation*}
		\begin{split}
			\eta(x)\Phi(x)&=\tilde{F}^{\alpha\bar{\beta}}(\tilde{U}(x_{0}))	\Big(\tilde{U}_{\alpha\bar{\beta}}(x)-\tilde{U}_{\alpha\bar{\beta}}(x_{0})\Big)-h(x)+h(x_{0})\\
			&\geq  \tilde{F}(\tilde{U}(x))- \tilde{F}(\tilde{U}(x_{0}))	-h(x)+h(x_{0}).
		\end{split}
	\end{equation*}
	Thus, $\Phi(x_{0})=0$ and $\Phi\geq 0$ near $x_{0}$ on $\partial\Omega$. Define
	\begin{equation*}
		\Psi:=-\sum_{j=1}^{ n}|X_{j}(u-\underline{u})|^{2}-Av+B\rho^{2} \qquad  \text{in}\ \Omega_{\delta}.
	\end{equation*}
	One can verify that  $\Phi+\Psi\geq 0 \  \text{on} \ \partial \Omega_{\delta}$ and
	$$L(\Phi+\Psi)\leq 0\qquad \text{in}\ \Omega_{\delta}$$  provided that $A\gg B\gg 1$. By Hopf's lemma, we know $Y_{n}\Phi(x_{0})\geq -C$, then $Y_{n}Y_{n}u(x_{0})\leq C$.
	
	Now we are in a position where all the eigenvalues of $U(x_{0})$ are bounded, so $\lambda(U)(x_{0})$  is contained in a compact subset of $\Gamma_{n}$. Since the operator $\lambda\mapsto\log\det(\lambda)$ is strictly increasing  with respect to each variable,
	\begin{equation*}
		\tilde{m}\geq m_{R}:=\log\det(\lambda'(\tilde{U}
		(x_{0})),R)-h(x_{0})>0
	\end{equation*}
	when $R$ is large enough. This proves (\ref{reduced goal}), and the proof is  complete.
\end{proof}

\section{Existence of subsolutions}
Suppose that $\Omega\subseteq M$ is a smooth pseudoconvex domain, and let $\rho$ be a strictly $J$-psh defining function for $\Omega$. Then there exists a uniform positive constant $\gamma>0$ such that $\ddbar \rho\geq \gamma \omega$. For each $s>0$, we set 
\begin{equation*}
	\underline{u}:=\hat{\varphi}+s(e^{\rho}-1),
\end{equation*}
where $\hat{\varphi}$ is an arbitrary $J$-psh extension of $\varphi|_{\partial\Omega}$. Then 
\begin{equation*}
	\begin{split}
		\ddbar \underline{u}=&\ddbar{\hat{\varphi}}+se^{\rho}(\ddbar{\rho}+\sqrt{-1}\partial\rho\wedge\bar{\partial}\rho)	\\
		\geq & 	s\gamma e^{\rho}\omega+se^{\rho}\sqrt{-1}\partial\rho\wedge\bar{\partial}\rho.
	\end{split}
\end{equation*}
Therefore,
\begin{equation*}
	\det(\underline{u}_{i\bar{j}})	\geq (s\gamma)^{n}e^{n\rho}(1+\frac{1}{\gamma}|\partial\rho|^{2}).
\end{equation*}
We may choose $s\gg 1$ such that $\det(\underline{u}_{i\bar{j}})\geq N:= \sup_{\bar{\Omega}}h$. Notice that $\underline{u}=\varphi$ on $\partial\Omega$, so $\underline{u}$ is a desired subsolution of Eq. (\ref{dp}).

~\\
{\bf Acknowledgments.} The author would like to thank his thesis advisor professor Xi Zhang for his constant support and  advice. The reaserch is  supported by the National Key R and D Program of China 2020YFA0713100.

\end{CJK}
\end{document}